\documentclass[12pt,psamsfonts]{amsart}
\usepackage{amssymb,eucal,epsfig}
\usepackage{verbatim,enumerate,mathdots}    
\usepackage[all,graph,frame]{xy}
\usepackage{amsthm}



\def\im{\operatorname{im}}
\def\ad{\operatorname{ad}}
\def\pip{\operatorname{\pi_\perp}}
\def\tr{\operatorname{tr}}
\def\Span{\operatorname{Span}}      
\def\End{\operatorname{End}}        

\def\sq{\subseteq}
\def\a{\mathfrak a}
\def\b{\mathfrak b}

\def\g{\mathfrak g}
\def\h{\mathfrak h}
\def\hp{{\mathfrak h}^{\perp}}

\def\m{\mathfrak m}

\def\p{\mathfrak p}
\def\r{\mathfrak r}
\def\s{\mathfrak s}
\def\t{\mathfrak t}
\def\z{\mathfrak z}

\def\N{\mathbb N}

\def\R{\mathbb R}

\def\0{\mathbb 0}

\def\<{\langle}                     
\def\>{\rangle}

\theoremstyle{plain}
\newtheorem{theorem}{Theorem}[section]
\newtheorem{proposition}[theorem]{Proposition}
\newtheorem{lemma}[theorem]{Lemma}

\newtheorem*{theorem*}{Theorem}

\theoremstyle{definition}
\newtheorem{Definition}[theorem]{Definition}
\newtheorem{remark}[theorem]{Remark}

\title{Totally Geodesic Subalgebras of Nilpotent Lie algebras}
\author[Cairns, Hini\'c Gali\'c and  Nikolayevsky]{Grant Cairns $^1$, Ana Hini\'c Gali\'c $^2$
 and Yuri Nikolayevsky $^3$}
 \begin{thanks}
{The authors are very thankful to Barry Jessup for his helpful comments.}
\end{thanks}

\address{$^1$
Department of Mathematics and Statistics\\
 La Trobe University\\
 Melbourne,  3086\\
 Australia}
\email{G.Cairns@latrobe.edu.au}

\address{$^2$
Department of Mathematics and Statistics\\
 La Trobe University\\
 Melbourne,  3086\\
 Australia}
\email{A.HinicGalic@latrobe.edu.au}

\address{$^3$
Department of Mathematics and Statistics\\
 La Trobe University\\
Melbourne,  3086\\
Australia}
\email{Y.Nikolayevsky@latrobe.edu.au}

\begin{document}

\maketitle

\begin{abstract}
A metric Lie algebra $\g$ is a Lie algebra equipped with an inner product. A subalgebra $\h$ of a metric Lie algebra $\g$   is said to be totally geodesic if the Lie subgroup corresponding to $\h$ is a totally geodesic submanifold relative to the left-invariant Riemannian metric defined by the  inner product, on the simply connected Lie group associated to $\g$. A nonzero element of $\g$ is called a geodesic if it spans a one-dimensional totally geodesic subalgebra. We give a new proof of Ka{\u\i}zer's theorem  that every metric Lie algebra possesses a geodesic. For nilpotent Lie algebras, we give several results on the possible dimensions of totally geodesic subalgebras.  We give an example of a codimension two totally geodesic subalgebra of the standard filiform nilpotent Lie algebra, equipped with a certain inner product. We prove that no other filiform Lie algebra possesses such a subalgebra. We show that in filiform nilpotent Lie algebras, totally geodesic subalgebras that leave invariant their orthogonal complements have dimension  at most half the dimension of the algebra. We give an example of a 6-dimensional filiform nilpotent Lie algebra that has no totally geodesic subalgebra of dimension $>2$, for any choice of inner product.
\end{abstract}

\section{Introduction}
Milnor's classic paper \cite{Mi} showed that the Riemannian geometry of Lie groups can be fruitfully investigated by working in the level of their Lie algebras. A number of recent papers have pursued this theme; see \cite{HL,Ch,KN,K}.  Our paper is  motivated by the work of Kerr and Payne \cite{KP}.

Consider a Lie group $G$ with finite dimensional, real, Lie algebra $\g$. Choose an inner product  $\langle \cdot,\cdot\rangle$ on $\g$ and consider the resulting  left-invariant Riemannian metric on $G$. Now let $\h$ be a subalgebra of $\g$ and consider the corresponding connected subgroup $H$ of $G$. One says that $\h$ is a {\em totally geodesic subalgebra} of $\g$ if $H$ is a totally geodesic submanifold of $G$. The interest of this notion is that the left cosets of $H$ define a totally geodesic foliation on $G$; see \cite{Mc,Gh,C1,C2,CG}.
There is a purely algebraic formulation of totally geodesic subalgebra condition, as we will now recall. A Lie algebra equipped with an inner product is called a {\em metric} Lie algebra. If $\g$ is a metric Lie algebra, with  inner product $\langle \cdot,\cdot\rangle$, the Levi-Civita connection $\nabla$ on $\g$  is given by the formula for left invariant vector fields (see \cite{KoNo}): for all $ X,Y,Z\in\g$,
\begin{equation} \label{e:Levi-Civita}
2\langle\nabla_XY,Z\rangle= \langle [X,Y]
,Z\rangle+\langle[Z,X] ,Y\rangle+\langle [Z,Y]
,X\rangle.
\end{equation}

\begin{Definition}
 A subalgebra $\h$ of  a metric Lie algebra $\g$ is said to be {\em totally geodesic}  if  $\nabla_Y Z\in\h$ for all $Y,Z\in \h$.
\end{Definition}

We are primarily interested in the nilpotent case, but let us first give some general results from which the nilpotent case can be better appreciated.
Firstly, recall the following well known equivalence.

\begin{lemma}\label{L:basic}
Let $\h$ be a subalgebra $\h$ of a metric Lie algebra  $\g$ and let $\hp$ denote the orthogonal complement of $\h$ in $\g$. Consider the linear map $\phi: \h^{\perp}\to \End(\h)$ defined by
$ \phi(X)(Y)=\pi_\h [X,Y]$ for all $X\in\hp, Y\in \h$, where $\pi_\h :\g\to\h$ denotes the orthogonal projection. Then $\h$ is a totally geodesic subalgebra of $\g$ if and only if the image of $\phi$ lies in the orthogonal Lie algebra $\mathfrak{so}(\h)$; that is, explicitly, if
\begin{equation}\label{L:TotGeod-cond}
\langle [X,Y],Z\rangle +\langle [X,Z],Y\rangle =0, \ \text{for all}\ X\in\hp, Y,Z\in \h.
\end{equation}
\end{lemma}

Notice that the condition for $\h$ to be totally geodesic depends only on the choice of orthogonal complement $\hp$, and on the inner product on $\h$; it is not affected by the choice of inner product on $\hp$.
Notice also that by the above lemma, $\h$ is  totally geodesic if $\phi\equiv 0$, that is, if the following condition is satisfied:

\begin{Definition} \label{D:inv}
If $\h$ is a subalgebra of a metric Lie algebra $\g$, we say that $\hp$ is {\em $\h$-invariant} if $[X,Y]\in \hp$, for all $X\in\hp,Y\in \h$.
\end{Definition}

Notice that given $\h$, the  $\h$-invariance of the orthogonal complement $\hp$ does not depend on the choice of inner product on $\h$ or on $\hp$.
Of course, $\hp$ is not $\h$-invariant for all totally geodesic subalgebras. For example, consider the solvable Lie algebra with orthonormal basis $\{X,Y,Z\}$ and relations $[X,Y]=Z, [X,Z]=-Y, [Y,Z]=0$. The subalgebra generated by $Y,Z$ is  totally geodesic, by the above lemma, but $\hp$ is not $\h$-invariant.
Nevertheless, $\h$-invariance is often useful and gives natural examples of totally geodesic subalgebras. For example, let $\g$ be a  Lie algebra, let $\r$ denote its radical and let $\s$ be a Levi subalgebra. So $\r$ is a solvable ideal, $\s$ is semisimple and $\g$ is a semidirect product: $\g=\s\ltimes \r$. Choose an inner product on $\g$ for which $\r$ is orthogonal to $\s$. Then as $\r$ is an ideal, $\r$ is $\s$-invariant, and so $\s$ is totally geodesic. Similarly, if $\g$ is a semisimple Lie algebra, consider a Cartan decomposition $\g=\t +\p$, where $[\t,\t]\subseteq \t, [\t,\p]\subseteq \p$ and $[\p,\p]\subseteq \t$; for example, see \cite{He} or \cite{Vi}. Choose an inner product on $\g$ for which $\t$ is orthogonal to $\p$.  Since $[\t,\p]\subseteq \p$, we have that $\p$ is $\t$-invariant,  and so $\t$ is totally geodesic.

If $Y$ is a nonzero element of a metric Lie algebra $\g$ and $\Span( Y) $ is totally geodesic, then by abuse of language, we say that $Y$ is a {\em geodesic}. We remark that some authors insist further that $Y$ have unit length; we do not impose this restriction.
Geodesics are also known as  {\em homogeneous geodesics}, to distinguish them from general geodesics on the underlying Lie group.
Extensions of this concept to homogeneous spaces have been examined in several papers; see
\cite{KS,KSE,AA}.

\begin{remark}\label{R:geo}
From Lemma \ref{L:basic}, a nonzero element $Y$ is a geodesic if and only if\break $\langle\nabla_Y Y,X\rangle=0$ for every $X\in\Span(Y)^\perp$.  But from (\ref{e:Levi-Civita}), $\langle\nabla_Y Y,Y\rangle=0$. Thus  $Y$ is a geodesic if and only if $\nabla_Y Y=0$. Moreover, from (\ref{e:Levi-Civita}), $\nabla_Y Y=0$ if and only if $\langle [X,Y],Y\rangle=0$ for all $X\in\g$. In other words, $Y$ is a geodesic if and only if $\hp$ is $\h$-invariant, where $\h=\Span( Y) $. \end{remark}

The following lemma is well known;  in each case, the orthogonal complement of the subalgebra is  obviously invariant  in the sense of Definition \ref{D:inv}.

\begin{lemma}\label{L:center}
Let $\g$ be a metric Lie algebra. Then
\begin{enumerate}[\rm (a)]
\item Every vector subspace of the centre $\z(\g)$ of $\g$ is a totally geodesic subalgebra of $\g$.
\item Every subalgebra that is orthogonal to the derived algebra $[\g,\g]$ of $\g$ is a totally geodesic subalgebra of $\g$.
\end{enumerate}
\end{lemma}

 Not surprisingly, the totally geodesic condition is very sensitive to the choice of inner product. For completeness, we prove the following result, which is due to Gotoh.

\begin{proposition}\label{P:any}\cite{Gotoh}
Suppose that a Lie algebra $\g$ has a subalgebra $\h$ that is totally geodesic for every inner product on $\g$. Then $\h$ is contained in the centre of $\g$.
\end{proposition}

Notice that Lemma \ref{L:center} shows that if $\g$ is a solvable metric Lie algebra, then for every inner product on $\g$, the nonzero vectors orthogonal to the derived algebra are geodesics. More generally, one has:

\begin{lemma}\label{L:geo}
If $\g$ is a  Lie algebra and $Y\in \g$ is nonzero, then there is an inner product on $\g$ for which $Y$ is a geodesic if and only if there does not exist $X\in\g$ with $[X,Y]=Y$.
\end{lemma}

For example, consider the solvable, unimodular Lie algebra $\g$ with basis $\{X,Y,Z\}$ and relations
\[
[X,Y]=Y, [X,Z]=-Z, [Y,Z]=0.\]
The above lemma shows that the element $Y$ is  not a  geodesic for any inner product on $\g$. In fact, it is uncommon, for a given inner product, that all nonzero vectors are geodesics.

\begin{proposition}\label{P:comp}
Suppose that $\g$ is a Lie algebra. The following conditions are equivalent:
\begin{enumerate}[\rm (a)]
\item There is an inner product on $\g$  for which every nonzero element is a geodesic.
\item $\g$ is isomorphic to the direct sum of an abelian Lie algebra and a semisimple Lie algebra of compact type.
\end{enumerate}
\end{proposition}

On the other hand, it is not difficult to see that every Lie algebra has an inner product for which there is at least one geodesic. For a given, fixed inner product,  Lemma \ref{L:center} shows that there is at least one geodesic when $\g$ is solvable.
The general case is less trivial. The following result is due to Ka{\u\i}zer  \cite{Ka}. We provide a different proof that uses classic topological results concerning maps and vector fields on spheres. Our proof is very similar to Du{\v{s}}ek's proof  of \cite[Theorem 7]{Du}, which is a more general result.

\begin{theorem}\label{T:somegeo}\cite{Ka}
Every metric Lie algebra possesses a geodesic.
\end{theorem}

For some results on the existence of bases whose elements are all geodesics, see \cite{CLNN}.

We now turn to the nilpotent case.  Recall that Mal\`cev's theorem \cite{Ma} says that when $G$ is nilpotent, $G$ possesses a uniform lattice if and only if there is a basis for $\g$ relative to which the commutator coefficients are integers. So this is a condition that is often evident in practice. When $G$ possesses a uniform lattice $\Gamma$, the set of right cosets of $\Gamma$ has a natural structure of a compact Riemannian manifold  $M$ for which the quotient map $G\to M$ is a Riemannian covering and the totally geodesic foliation defined by $\h$ on $G$ descends to give a  totally geodesic foliation on $M$. This construction provides important standard  examples of totally geodesic foliations on compact manifolds.

The following fact is an immediate consequence of Lemma \ref{L:geo}.

\begin{lemma}\label{L:geodesic}
If $\g$ is a nilpotent Lie algebra and $Y\in \g$ is nonzero, then there is an inner product on $\g$ for which $Y$ is a geodesic.
\end{lemma}

For a given inner product, most nilpotent Lie algebras have more geodesics than those provided by Lemma \ref{L:center}.

\begin{proposition}\label{P:heis}
Suppose that $\g$ is a nonabelian nilpotent Lie algebra of dimension $n$. The following conditions are equivalent:
\begin{enumerate}[\rm (a)]
\item For each inner product on $\g$, the only geodesics are the vectors orthogonal to the derived algebra $[\g,\g]$ and those contained in the centre $\z(\g)$ of $\g$.
\item $\g$ is 2-step nilpotent and for each $X\in\g$ with $X\not\in\z(\g)$, the adjoint map $\ad(X)$ is surjective onto $[\g,\g]$.
\end{enumerate}
\end{proposition}

Further results for totally geodesic subalgebras 2-step nilpotent Lie algebras are given by Eberlein \cite{Eb}. Condition (b) of the above proposition is a slightly weaker version of Eberlein's  {\em nonsingularity}  condition, which requires that for all $X\not\in\z(\g)$, the map $\ad(X)$ is surjective onto $\z(\g)$. In particular, for nonsingular 2-step nilpotent Lie algebras with one dimensional centre,   \cite[Corollary 5.7]{Eb} gives a classification of totally geodesic subalgebras and extends the above proposition for these algebras. An example of an algebra with\break $2$-dimensional centre satisfying condition (b) of the above proposition is the 6-dimensional algebra with basis $\{X_1,X_2,X_3,X_4,Y_1,Y_2\}$ and nontrivial relations
\[
[X_1,X_2]=-[X_3,X_4]=Y_1,\quad
[X_1,X_3]=[X_2,X_4]=Y_2.
\]

In general, if a metric Lie algebra $\g$ has an orthonormal basis $\{X_1,\dots,X_n\}$, then the criterion $\nabla_Y Y=0$ for a vector $Y=\sum_{i=1}^n a_iX_i$ to be a geodesic is quadratic in the coefficients $a_i$. So the set of geodesics is a real semi-algebraic variety, which can be quite large as the following proposition shows.

\begin{proposition}\label{P:4D1Dtga}
Let $\g$ be $4$-dimensional metric nilpotent Lie algebra with a two-dimensional derived algebra. Then
    there is an orthonormal basis  $\{X_1,X_2,X_3,X_4\}$ such that the defining relations for $\g$ are
        \[
            [X_1,X_2] =\alpha X_3+\beta X_4,\quad [X_1,X_3] = \gamma X_4,
        \]
    where $\alpha,\beta,\gamma\in \R$ and $\alpha,\gamma>0$. Furthermore,
\begin{enumerate}[\rm (a)]
    \item
    the set of geodesic vectors is the set of nonzero vectors in the union of the subspace $\Span( X_1,X_2)$ orthogonal to the derived algebra and
    the cone $\{X=xX_2+yX_3+zX_4 : \alpha xy + \beta xz + \gamma yz=0\} \subseteq \Span(X_2,X_3,X_4)$.
\item if $\g$ has a $2$-dimensional totally geodesic subalgebra $\h$, then $\beta =0$ and in this case, either
$\h=\Span(X_2,X_4)$ or $\h=\Span(\gamma X_2-\alpha X_4,X_3)$.\end{enumerate}
\end{proposition}

For nilpotent Lie algebras, the existence of a totally geodesic subalgebra of low codimension is a strong restriction. We prove:

\begin{proposition}\label{P:cod1}
Let $\g$ be a nilpotent metric Lie algebra and suppose that $\h$ is a totally geodesic subalgebra of $\g$ of codimension 1. Then $\g$ is a direct sum of Lie ideals, $\g\cong \h\oplus \R$.
\end{proposition}

A key subfamily of the variety of nilpotent Lie algebras is provided by the filiform algebras \cite{GK}. Recall that a nilpotent Lie algebra $\g$ of dimension $n$ is said to be {\em filiform} if it possesses an element of maximal nilpotency; that is, there exists $X\in\g$  with $\ad^{n-2}(X)\not=0$, where $\ad(X):\g\to\g$ is the adjoint map $\ad(X)(Y)=[X,Y]$. The simplest example of a filiform Lie algebra is the {\em standard} filiform algebra $L_n$, which has  a basis $\{X_1,\dots,X_n\}$, whose only nonzero relations are $[X_1,X_i]=X_{i+1}$, for $i=2,\dots,n-1$. We will refer to this as the {\em standard  filiform metric} algebra when the basis $\{X_1,\dots,X_n\}$ is taken to be orthonormal. In this case, the subalgebra generated by the $X_i$ with $i$ even, is a totally geodesic subalgebra.  Kerr and Payne have shown that $L_n$ has no totally geodesic proper subalgebras of greater dimension.

\begin{theorem}\label{T:KP}\cite[Theorem 6.4]{KP}
If $\h$ is a totally geodesic proper subalgebra of the standard filiform metric Lie algebra $L_n$, then $\dim( \h)\leq n/2$.
\end{theorem}

\begin{remark}
The result of \cite[Theorem 6.4]{KP} is slightly more general that the one we have presented above. Let $C=\{c_2,\dots,c_{n-1}\}$ be an ordered list of nonzero reals, possibly with repetitions.
Let $L_C$ denote the  nilpotent filiform metric Lie algebra with orthonormal basis $\{X_1,\dots,X_n\}$, whose only nonzero relations are $[X_1,X_i]=c_iX_{i+1}$, for $i=2,\dots,n-1$. It is shown in \cite[Theorem 6.4]{KP}  that  if $\h$ is a totally geodesic proper subalgebra of $\g$, then $\dim (\h)\leq (\dim(\g))/2$.
In fact, it is easy to see that if $L_C$ possesses a totally geodesic  subalgebra of dimension $k$, then so too does the metric Lie algebra $L_n$. We give details in Remark \ref{R:C=L} in Section \ref{S:ex} below.
 \end{remark}

For arbitrary filiform nilpotent Lie algebras, the following result shows that the conclusion of the Kerr-Payne theorem also holds for a natural class of totally geodesic subalgebras.

\begin{theorem}\label{T:fili}
Suppose that $\g$ is a filiform nilpotent metric Lie algebra, $\h$ is a proper subalgebra of $\g$ and $\hp$ is $\h$-invariant. Then $\dim (\h)\leq (\dim(\g))/2$.
\end{theorem}

By Proposition \ref{P:cod1}, filiform nilpotent metric Lie algebras do not possess totally geodesic subalgebras of codimension one.  In contrast to the Kerr-Payne theorem and Theorem \ref{T:fili}, we prove:

\begin{theorem}\label{T:cd2f}
For all $n\geq 3$, the standard filiform Lie algebra $L_n$ possesses an inner product relative to which $L_n$ has a totally geodesic subalgebra of codimension two.
\end{theorem}

Conversely we have:

\begin{theorem}\label{T:conv}
If a filiform nilpotent metric Lie algebra $\g$ of dimension $n$ possesses  a totally geodesic subalgebra of codimension two, then $\g$ is isomorphic to the standard filiform Lie algebra $L_n$.
\end{theorem}

In fact, some filiform Lie algebras have the property that they only have  totally geodesic subalgebras of low dimension, regardless of the choice of inner product.

\begin{theorem}\label{T:dim6}
Consider the following $6$-dimensional filiform Lie algebra $\g$:
\begin{align*}
[X_1,X_i]&=X_{i+1},\quad {\text{for\ }} i=2,\dots,5,\\
[X_2,X_3]&=-X_6.
\end{align*}
For no choice of inner product, does $\g$ possess a totally geodesic subalgebra of dimension greater than two.
\end{theorem}

Further results on filiform nilpotent metric Lie algebras are given in \cite{CHGN}, which is the sequel to the present paper.

The paper is organised as follows.  In Section \ref{S:prel} we establish the  results for general Lie algebras: Lemmas \ref{L:basic} and \ref{L:geo},  Propositions   \ref{P:any} and  \ref{P:comp}, and Theorem  \ref{T:somegeo}. In Section \ref{S:nil}  we prove Proposition \ref{P:heis} and  \ref{P:cod1}, for nilpotent Lie algebras, and we establish some general results that we use later in the paper: Lemma \ref{L:end}, Proposition \ref{P:nil} and  Lemma \ref{L:cod2}. In Section \ref{S:fa} we discuss  filiform Lie algebras. We state Vergne's theorem (Theorem \ref{T:Ve}), and after establishing some preliminary facts (Lemmas \ref{L:Maxnilp},  \ref{L:mn1} and \ref{L:mn2}),  we give the proofs of   Theorems \ref{T:fili} and \ref{T:conv}. Section \ref{S:ex} treats the examples of Proposition \ref{P:4D1Dtga}, and Theorems \ref {T:cd2f} and \ref{T:dim6}. Finally, in the Appendix we give a proof of Vergne's theorem that we used in Section \ref{S:fa}.

Throughout this paper, all Lie algebras will be assumed to have finite dimension, and be defined over the reals. If $\h$ is a subalgebra of a metric Lie algebra $\g$, we denote the orthogonal complement of $\h$ by $\hp$. We use the symbol $\oplus$ to denote the direct sum of Lie ideals. Given subsets $V_1,\dots,V_k$ of $\g$, we denote the vector subspace that they span by $\Span(V_1,\dots,V_k)$, and given elements $X_1,\dots,X_k$, we write $\Span(X_1,\dots,X_k)$ instead of $\Span(\{X_1,\dots,X_k\})$.
For a given basis $B=\{X_1,\dots,X_n\}$ of $\g$,  we write $\g_k=\Span( X_k,\dots,X_n)$.
For $Y\in\g$, we define the $B$-degree of $Y$, denoted $\deg_B(Y)$, to be the largest natural number $k$ such that $Y\in\g_k$, and  for convenience, we set $\deg_B(0)=\infty$.

\section{Preliminaries}\label{S:prel}

\begin{proof}[Proof of Lemma \ref{L:basic}]
If $Y,Z\in\h$, then as $\h$ is a subalgebra, $[Y,Z]\in\h$. Then, by definition, $\h$ is a totally geodesic subalgebra if and only if  for all $X\in\hp,Y,Z\in\h$ we have
\begin{align*}
0=2\langle \nabla_Y Z,X\rangle &=\langle [Y,Z],X\rangle +\langle [X,Y],Z\rangle +\langle [X,Z],Y\rangle  \\
&=\langle [X,Y],Z\rangle +\langle [X,Z],Y\rangle ,
\end{align*}
since $\langle [Y,Z],X\rangle =0$.
\end{proof}

\begin{proof}[{Proof of Proposition \ref{P:any}}]  Suppose that  $\h$  is a   codimension $k$  subalgebra  of $\g$ that is totally geodesic for every inner product on $\g$.  First we fix an orthogonal complement $\hp$ to $\h$, and we will vary the inner product on $\h$.  By Lemma \ref{L:basic}, as  $\h$  is totally geodesic for every inner product,  $ [X,Y]$ is perpendicular to $Y$, for all $X\in\hp, Y\in \h$ and for every inner product on $\h$. But if this holds  for every inner product on $\h$, then $[X,Y]$ must have zero component in $\h$; in other  words, $\hp$ is $\h$-invariant. Thus $\hp$ is $\h$-invariant for every choice of orthogonal complement $\hp$.

 Let $\{X_1,\dots,X_k\}$ be an arbitrary basis for $\hp$ and let $Y\in \h$. Since $\hp$ is $\h$-invariant, for $i=1,\dots,k$, we have $[X_i,Y]=\sum_{j =1}^ka_{ij}X_j$, for some constants $a_{ij}$. We now vary the orthogonal complement to $\h$. Choose constants $c_1,\dots,c_{k}$ and consider a new orthogonal complement spanned by
the vectors $c_1Y+X_1,c_2Y+X_2,\dots,c_{k}Y+X_{k}$. As this new orthogonal complement  is also $\h$-invariant, we have
\[
[c_iY+X_i,Y]=\sum_{j=1}^{k} b_{ij}(c_jY+X_j),
\]
for some constants $b_{ij}$.  Then as $[c_iY+X_i,Y]=[X_i,Y]$, the linear independence of $X_1,\dots,X_k,Y$ gives $b_{ij}=a_{ij}$ for all $1\leq i,j\leq k$, and $\sum_{j=1}^{k} a_{ij}c_j=0$ for all $1\leq i$. Since this is true for all choices of $c_j$, we have $a_{ij}=0$ for all $1\leq i,j\leq k$. Thus $[X_i,Y]=0$ for each $i$ and hence $[X,Y]=0$ for all $X\in\hp$. Now let $Z\in \h$ and consider a third orthogonal complement spanned by the basis $\{Z+X_1,X_2,\dots,X_k\}$; applying the above argument,  $[Z+X_1,Y]=0$, and so $[Z,Y]=0$. Hence $Y\in \z(\g)$, for all  $Y\in \h$. Thus $\h \subseteq  \z(\g)$.
\end{proof}

\def\proofname{Proof of Proposition \ref{P:comp}}
\begin{proof} First suppose that $\g$  has an inner product $\langle \cdot,\cdot\rangle$ for which every nonzero element is a geodesic. If $Y\in\g$, then as $Y$ is either zero or is a geodesic, we have   by Remark \ref{R:geo}, $\langle [X,Y],Y\rangle=0$ for all $X\in\g$. In particular, for all $X,Y,Z\in\g$ we have $\langle [X,Y+Z],Y+Z\rangle=\langle [X,Y],Y\rangle=\langle [X,Z],Z\rangle=0$, from which we obtain $\langle [X,Y],Z\rangle+\langle [X,Z],Y\rangle=0$; that is, $\ad(X)$ is skew-adjoint. It follows that the Riemannian metric determined by $\langle \cdot,\cdot\rangle$ on the  simply connected Lie group associated with $\g$ is bi-invariant \cite[Lemma 7.2]{Mi}, and consequently $\g$ is the direct sum of an abelian Lie algebra and a semisimple Lie algebra of compact type \cite[Lemma 7.5]{Mi}.

Conversely, if $\g$ is the direct sum of an abelian Lie algebra and a semisimple Lie algebra of compact type, then choose an inner product for which the two factors are orthogonal, and on the semisimple factor, take the inner product defined by the Killing form. For this inner product the adjoint maps $\ad(Y)$ are skew adjoint. In particular, $\langle [X,Y],Y\rangle=0$ for all $X,Y\in\g$. Consequently, from Remark \ref{R:geo}, $Y$ is a geodesic for all nonzero $Y\in\g$.
\end{proof}

\begin{proof}[{Proof of Lemma \ref{L:geo}}]
If $\g$ is a metric Lie algebra and $Y\in\g$ is a geodesic, then  by Remark \ref{R:geo}, for all $X\in\g$, the element $[X,Y]$ is perpendicular to $Y$, and so in particular,  $[X,Y]\not=Y$.
Conversely, if $Y\in \g $ is nonzero and $[X,Y]\not=Y$ for all $X\in\g$, then $Y$ does not belong to the image $\im(\ad(Y))$ of the adjoint map $\ad(Y)$. Choose an inner product on $\g$ for which $Y$ is orthogonal to  $\im(\ad(Y))$. Then by Remark \ref{R:geo}, $Y$ is a geodesic. \end{proof}

\begin{proof}[{Proof of Theorem \ref{T:somegeo}}]
Let $\g$ be a  Lie algebra of dimension $n$ with inner product $\langle \cdot,\cdot\rangle$. For each $Y\in\g$, let $\alpha_Y$ denote the linear 1-form on $\g$ defined by $\alpha_Y(X)= \langle [X,Y],Y\rangle$ for all $X\in \g$. By Remark \ref{R:geo}, if $Y$ is nonzero, then $Y$ is a geodesic if and only if $\alpha_Y=0$.  For each $Y\in\g$, let $f(Y)$ denote the vector dual to $\alpha_Y$; that is, $\langle f(Y),X\rangle=\alpha_Y(X)$ for all $X\in \g$.  Suppose that $\g$ has no geodesics; so $f(Y)\not=0$ for all nonzero $Y\in\g$. Restricting to the unit sphere $S^{n-1}$ in $\g$, consider the continuous function $\hat f:S^{n-1}\to S^{n-1}$ defined by $\hat f(Y)=f(Y)/\vert f(Y)\Vert$. Notice that for each $Y\in S^{n-1}$, we have $\langle f(Y),Y\rangle=\alpha_Y(Y)=\langle [Y,Y],Y\rangle=0$. Thus, for each $Y\in S^{n-1}$, the vectors $Y$ and $\hat f(Y)$ are orthogonal. Consequently the assignment $Y\mapsto \hat f(Y)$ defines a continuous nowhere zero vector field on $S^{n-1}$. This is impossible if $n$ is odd, because the Euler characteristic of the even dimensional spheres are nonzero (see for example, \cite{Br}). Our treatment of the case where $n$ is even uses related topological ideas. Notice that since for each $Y\in S^{n-1}$, the vectors $Y$ and $\hat f(Y)$ are orthogonal, $\hat f$ sends no point to its antipode. It follows that $\hat f$ is homotopic to the identity map, and in particular, $\hat f$ has degree 1 (see \cite[Corollary 4(a)]{Wh}). But by definition, we have $\hat f(-Y)=\hat f(Y)$. Thus the map $\hat f$ factors through the real projective space $P^{n-1}$, and consequently $\hat f$ has even degree (see \cite[Theorem 2]{Wh}). This is a contradiction; we conclude that $\g$ has a geodesic.
\end{proof}

\section{Nilpotent Lie algebras}\label{S:nil}

\begin{proof}[{Proof of Proposition \ref{P:heis}}]
First assume that condition (a) holds. By Lemma \ref{L:geodesic}, every element of $\g$ either lies in $\z(\g)$, or does not lie in $[\g,\g]$, and thus  $[\g,\g]\subseteq \z(\g)$.  In particular,  $\g$ is 2-step nilpotent.
Suppose that $X\in\g\backslash\z(\g)$ and that $Y\in[\g,\g]$ with $Y\not\in\im(\ad(X))$.
Choose an inner product on $\g$ for which $X,Y$ have unit length, $X$ is orthogonal to $\z(\g)$ and $Y$ is orthogonal to $\im(\ad(X))$. In particular, $X$ is orthogonal to the derived algebra $[\g,\g]$ and hence by  Lemma \ref{L:center}, $X$ is a geodesic. Extend $X,Y$ to an orthonormal basis $\{X,Y,Z_1,\dots,Z_{n-2}\}$ of $\g$. Since $X$ is a geodesic,  $[X,Z_i]$ is perpendicular to $X$ for each $i=1,\dots,n-2$. So, as  $Y$ is orthogonal to $\im(\ad(X))$, we have $[X,Z_i]\in\Span( Z_1, \dots,Z_{n-2})$, for each $i=1,\dots,n-2$. Now consider $X+Y$. The orthogonal complement to $X+Y$ is spanned by $X-Y,Z_1, \dots,Z_{n-2}$. As $Y$ belongs to the centre, we have $[X+Y,X-Y]=0$. Moreover,  for each $i=1,\dots,n-2$, we have $[X+Y,Z_i]=[X,Z_i]\in \Span( Z_1, \dots,Z_{n-2})$ and so  $[X+Y,Z_i]$ is orthogonal to $X+Y$. So for $\h=\Span( X+Y)$, we have that $\hp$ is $\h$-invariant, and hence $X+Y$ is a geodesic. So, by hypothesis,  $X+Y$ is either orthogonal to  $[\g,\g]$ or contained in  $\z(\g)$. But $X+Y$ does not belong to $\z(\g)$, since $Y\in\z(\g)$ and $X\not\in\z(\g)$, and $X+Y$ is not orthogonal to $[\g,\g]$ since $X$ is orthogonal to $[\g,\g]$ and $Y\in[\g,\g]$. Thus $\ad(X)$ is surjective onto $[\g,\g]$.

Conversely, suppose that $\g$ is 2-step nilpotent and for each $X\in\g$ with $X\not\in\z(\g)$, the adjoint map $\ad(X)$ is surjective onto $[\g,\g]$. Choose an inner product $\langle \cdot,\cdot\rangle$ on $\g$ and suppose that $X$ is a geodesic in $\g$ and that $X\not\in\z(\g)$. Let $X'$ denote the component of $X$ in $[\g,\g]$. By hypothesis, there exists $Y\in\g$ with $[X,Y]=X'$. Since $X$ is a geodesic, we have $\langle[X,Y],X\rangle=0$, by Remark \ref{R:geo}, and hence $X'$ is orthogonal to $X$. But this gives $X'=0$, and so $X$ is orthogonal to $[\g,\g]$.
 \end{proof}

Suppose that $\g$ is a Lie algebra with an inner product $\langle \cdot,\cdot\rangle$ for which $\h$ is a totally geodesic subalgebra.  Let $\pi: \g\to \h$ and $\pip: \g\to \hp$ be the orthogonal projections. Consider the map $\psi: \h\to \End(\hp)$ defined by $\psi(Y): X\mapsto \pip[Y,X]$. We will show below that $\psi$ is a Lie algebra homomorphism. If $\h$ is nilpotent, then its image  $\psi(\h)$ is a nilpotent subalgebra. Moreover, a Lie algebra is nilpotent if and only if the adjoint map of each element is a nilpotent map. However, it is not in general true that if $\h$ is nilpotent and $V$ is a vector space, and if $f:\h\to \End(V)$ is a Lie algebra representation, then the elements $f(Y)$ are necessarily nilpotent maps for all $Y\in\h$; cf.~Zassenhaus' result \cite[p.~41]{Ja}. Nevertheless, in our particular case, we have:

\begin{lemma}\label{L:end}
The function $\psi$ is a Lie algebra homomorphism and if $\g$ is nilpotent, then  $\psi$ is a nilpotent representation of $\h$ on $\hp$.
\end{lemma}

\def\proofname{Proof}
\begin{proof}
Let $X\in\hp,Y,Z\in\h$. We have
\begin{align*}
\psi([Y,Z])(X)&=\pip[[Y,Z],X]\\
&=\pip([Y,[Z,X]] - [Z,[Y,X]])  &(\text{by Jacobi})\\
&=\pip[Y,\psi(Z)X]-\pip[Z,\psi(Y)X]  &(\text{as $\h$ is a subalgebra})\\
&=(\psi(Y)\circ\psi(Z)-\psi(Z)\circ \psi(Y))(X).
\end{align*}
So $\psi$ is a Lie algebra representation. If $\g$ is nilpotent, then to see that the map $\psi(Y)$ is nilpotent, it suffices to show that for all $k\in\N$, one has $(\psi(Y))^k(X)=\pip (\ad(Y))^k(X)$. We have
\begin{align*}
(\psi(Y))^k(X)&=\pip[Y,(\psi(Y))^{k-1}(X)]      &(\text{by definition})\\
&=\pip[Y,\pip (\ad(Y))^{k-1}(X)]    &(\text{by the inductive hypothesis})\\
&=\pip[Y,(\ad(Y))^{k-1}(X)]    &(\text{as $\h$ is a subalgebra})\\
&=\pip (\ad(Y))^k(X),
\end{align*}
as claimed.
\end{proof}

\begin{proposition}\label{P:nil}
Suppose that $\g$ is a Lie algebra and consider the natural quotient map $p:\g\to\g/ [\g,\g]$. If $\g$ is nilpotent, then $\g$ has no proper subalgebra $\a$ with $p(\a)=\g/[\g,\g]$.
\end{proposition}

\begin{proof} Let $C_i(\g)$ be the lower central series for $\g$, that is, $C_0(\g)=\g$ and $C_i(\g)=[\g,C_{i-1}(\g)]$, for $i \ge 1$, and let $C_i(\a)$ be the
lower central series for $\a$. Seeking a contradiction, suppose that such an algebra $\a$ exists, that is, $\g=\Span(\a,C_1(\g))$. So
$C_1(\g)=[\Span(\a,C_1(\g)), \Span(\a,C_1(\g))] \sq \Span(C_1(\a),C_2(\g))$, and then by induction, $C_i(\g) \sq \Span(C_i(\a),C_{i+1}(\g))$. As the
opposite inclusion is obvious, we have $C_i(\g) =  \Span(C_i(\a),C_{i+1}(\g))$, for all $i \ge 0$. Let $N > 0$ be such that $C_N(\g) \ne 0 = C_{N+1}(\g)$.
Then $C_N(\g) = C_N(\a)$, so $C_{N-1}(\g) =  \Span(C_{N-1}(\a),C_{N}(\g))=\Span(C_{N-1}(\a),C_{N}(\a))=C_{N-1}(\a)$, and then by induction, $C_0(\g)=C_0(\a)$, a contradiction.
\end{proof}

\begin{proof}[{Proof of Proposition  \ref{P:cod1}}]
Suppose that $\h$ has codimension 1. Let $X$ be a unit vector with $\hp=\Span(X)$. Consider the map $\psi : \h\to \End(\hp)$ defined above. For each $Y\in\h$, Lemma \ref{L:end} says that $\psi(Y)$ is nilpotent and consequently, as $\dim( \hp)=1$, we have $\psi(Y)=0$. Hence $\psi\equiv 0$, and thus $\h$ is an ideal of $\g$. In particular, $\ad(X) :\h \to\h$ is a Lie algebra derivation, and since $\h$ is nilpotent, so is $\ad(X)$. It remains to show that $\ad(X)\equiv 0$.
Arguing by contradiction, assume that $\ad(X)(Y)\neq 0$ for some $Y\in\h$. As $\ad(X)$ is  nilpotent, and we can choose $Y$ such that $\ad(X)(Y)\neq 0$  and  $\ad(X)(\ad(X)(Y))=0$. Let $Z=\ad(X)(Y)$.
Then
\begin{align*}
\langle [X,Y],Z\rangle +\langle [X,Z],Y\rangle &=\langle \ad(X)(Y),Z\rangle +\langle \ad(X)(Z),Y\rangle\\
& =\|Z\|^2+0\neq 0,
\end{align*}
which is in contradiction with condition \eqref{L:TotGeod-cond}. Thus $\ad(X)\equiv 0$ and hence $\hp$ is an ideal. So $\g=\h\oplus\hp\cong \h\oplus \R$.
\end{proof}

\begin{lemma}\label{L:cod2}
Suppose that $\g$ is a nilpotent metric Lie algebra and $\h$ is a totally geodesic subalgebra of $\g$ of codimension two.
\begin{enumerate}[\rm (a)]
\item Let $\{X_1,X_2\}$ be a  basis for $\hp$. Then
\begin{enumerate}[\rm (i)]
\item Up to nonzero factor, the vector $[X_1,X_2]$ is independent of $X_1,X_2$,
\item $[X_1,X_2]\in \h$,
\item $[X_1,X_2]\in \z(\h)$.
\end{enumerate}
\item Let $\a$ be the intersection of the subalgebras of $\g$ that contain $\hp$. Then $\a\cap \h\subseteq \z(\h)$.
\end{enumerate}
\end{lemma}

Notice that in the above lemma, part (a)(ii) is redundant, since it follows immediately from part (a)(iii). We include part (a)(ii) to emphasise this useful fact.

\begin{proof}(a)(i). If $\{X_1',X_2'\}$ is another  basis for $\hp$, then
\[
\begin{pmatrix}
X_1'\\X_2'
\end{pmatrix}
=
A\begin{pmatrix}
X_1\\X_2
\end{pmatrix}
\]
for some $2\times 2$ nonsingular matrix $A$. Hence $[X_1',X_2']=\det(A) [X_1,X_2]$.

By part (a)(i) we may choose $X_1,X_2$ to simplify the calculations in the rest of this lemma. Consider the natural projection map $p:\g\to \g/[\g,\g]$. By Proposition \ref{P:nil}, we have $p(\h)\not=\g/[\g,\g]$. Choose a codimension one vector subspace $\b$ of $\g/[\g,\g]$ containing $p(\h)$. Note that $\b$ is an ideal of the abelian Lie algebra $\g/[\g,\g]$. Let $\m$ denote the preimage by $p$ of $\b$, so $\m$ is a codimension one  ideal of $\g$ and $\m$ contains $\h$ as codimension one subalgebra.  Let $X_1\in \g$ be a nonzero vector orthogonal to $\m$, and let  $X_2\in \m$  be a nonzero vector orthogonal to $\h$.  Note that $\h$ is a codimension one totally geodesic subalgebra of $\m$, so by Proposition \ref{P:cod1}, $\m$ is the direct sum of Lie ideals, $\m=\h\oplus \Span(X_2)$.

(a)(ii). Since $\h$ is totally geodesic,  Lemma \ref{L:basic} gives $\langle [X_1,Y],Y\rangle=0$ for each $Y\in \h$. Since $\g$ is nilpotent and therefore unimodular, $\tr(\ad(X_1))=0$, and so  $\langle [X_1,X_2],X_2\rangle=0$.  Since $\m$ is an ideal  $[X_1,X_2]\in \m$. Thus $[X_1,X_2]\in \h$.

(a)(iii). Because of the direct sum structure, we have $[X_2,Y]=0$ for all $Y\in \h$. Let $Y\in \h$.  So the Jacobi identity gives $[[X_1,X_2],Y]=[[X_1,Y],X_2]$. Since $\m$ is an ideal  $[X_1,Y]=\alpha X_2+ Z$ for some $\alpha\in\R,Z\in \h$. Thus $[[X_1,Y],X_2]=[\alpha X_2+ Z,X_2]=[Z,X_2]=0$.

(b). Define  subspaces $V_i$ of $\g$ inductively by posing $V_0=\Span(X_1,X_2),V_1=\Span([X_1,X_2])$ and $V_{i+1}=\Span([X_1,Y]: Y\in V_i)$ for all $i\geq 2$. Let
\[
V=\Span(V_i :i\geq0)
\]
We claim that  $ V_i\subseteq \z(\m)$, for each $i\geq 1$. From part (a)(iii), we have $V_1\subseteq \z(\m)$. Suppose that  $ V_i\subseteq \z(\m)$ and let $Z\in V_{i+1}$.  Thus $Z=[X_1,Y]$ for some $Y\in V_i$.  Let $W\in \m$.  As $Y\in V_i\subseteq \z(\m)$, we have $[Y,W]=0$. As $\m$ is an ideal,  $[X_1,W]\in \m$ and so $[[X_1,W],Y]=0$.
Thus the Jacobi identity gives
\[
[Z,W]=[[X_1,Y],W]=[[X_1,W],Y]+[X_1,[Y,W]]=0.
\]
Thus $V_{i+1}\subseteq \z(\m)$, and by induction, $\Span(V_i :i\geq1)\subseteq \z(\m)$. In particular,
\[
\h\cap V= \h\cap\Span(V_i :i\geq1) \subseteq \z(\h)\]
 It remains to show that $V=\a$. Clearly, every subalgebra that contains $V_0$, must also contain $V_i$, for all $i\geq 1$, and hence it must contain $V$. So it suffices to show that $V$ is  a subalgebra. Notice that by construction $[X_1,V]\subseteq V$. Suppose that $Y_1,Y_2\in V$. Then
for $j=1,2$ we may write $Y_j=\alpha_j X_1 + Z_j$, where $\alpha_j\in\R,Z_j\in \Span(X_2,V_i :i\geq1)$. Thus, since  $\Span(V_i :i\geq1)\subseteq \z(\m)$ and $X_2\in \z(\m)$,
\[
[Y_1,Y_2]=\alpha_1[X_1,Z_2]-\alpha_2[X_1,Z_1]\in V.
\]
Hence $V$ is  a subalgebra, as claimed.\end{proof}

\begin{remark}\label{R:irreg6}
In general, Lemma \ref{L:cod2}(b) is not true in codimension $>2$. For example, consider the $6$-dimensional metric Lie algebra $\g$ having  a basis $\{X_1,\dots,X_6\}$ and presentation
\[
[X_2,X_5]=X_6,\ [X_3,X_4]=-X_6,\ [X_1,X_i]=X_{i+1},\ \text{for}\ i=2,3,4,5.
\]
Let $E_1=X_1-X_2$ and $E_i=X_i$ for $i=2,\dots,6$, and choose the inner product on $\g$ for which $E_1,\dots,E_6$ are orthonormal. Notice that the codimension  3 subalgebra $\h=\Span(E_2,E_5,E_6)$ is totally geodesic. Here $\hp=\Span(E_1,E_3,E_4)$. The intersection  of the subalgebras of $\g$ that contain $\hp$ is $\a=\Span(E_1,E_3,E_4,E_5,E_6)$, so $\a\cap \h= \Span(E_5,E_6)$, and  $\a\cap \h\not\subseteq \z(\h)$.
\end{remark}

\section{Filiform Algebras}\label{S:fa}

Recall that a nilpotent Lie algebra $\g$ of dimension $n$ is said to be {\em filiform} if it possesses an element of maximal nilpotency; that is, there exists $X\in\g$  with $\ad^{n-2}(X)\not=0$. The following theorem is due to Mich\`ele Vergne \cite{Ve}. We give a proof in the appendix.

\begin{theorem}\label{T:Ve}
Every filiform  nilpotent Lie algebra has a basis $\{X_1,\dots,X_n\}$ such that
\begin{align*}
[X_1,X_i]&=X_{i+1},\ \text{for all } i\geq 2,\\
[X_i,X_j]&\in \g_{i+j}, \ \text{for all } i,j\ \text{with }i+j\not= n+1,\\
\exists\alpha\in\R \ \text{such that }[X_i,X_{n-i+1}]&=(-1)^i \alpha X_n, \ \text{for all }2\leq i\leq n-1,\\
\text{if }n\ \text{is odd, }\alpha&=0,
\end{align*}
where $\g_k=\Span( X_k,\dots,X_n)$ and for convenience we have set $X_i=0$ for $i>n$.
\end{theorem}

\begin{Definition}  We will say that a filiform  nilpotent Lie algebra $\g$ is {\em regular} if $\g$ has a basis $\{X_1,\dots,X_n\}$  satisfying the conditions of Theorem \ref{T:Ve}, with $\alpha=0$.
\end{Definition}

\begin{remark}\label{R:dim6}
By definition, a filiform  nilpotent Lie algebra $\g$ is regular if and only if $\g$ has a basis $\{X_1,\dots,X_n\}$  with $[X_i,X_j]\in \g_{i+j}$ for all $i,j$. By  Theorem \ref{T:Ve},   $\g$ is regular  if it has odd dimension. Notice that for the 4-dimensional algebra with basis $\{X_1,\dots,X_4\}$ and nontrivial presentation
\[
[X_1,X_2]=X_3,\ [X_1,X_3]=X_4,\ [X_2,X_3]=X_4,
\]
the relations are those of Theorem \ref{T:Ve}, with $\alpha=1$. Nevertheless, this algebra is regular; it is isomorphic to the standard filiform algebra $L_4$, as one can see by examining the basis $\{X_1,X_2-X_1,X_3,X_4\}$. Indeed, there are no irregular filiform  nilpotent Lie algebras in dimension less than $6$.
The 6-dimensional Lie algebra of Remark \ref{R:irreg6} is an irregular  filiform algebra.
\end{remark}

\begin{lemma}\label{L:Maxnilp}
Let $\g$ be a filiform  nilpotent  Lie algebra of dimension at least $5$, having a basis $\{X_1,\dots,X_n\}$  satisfying the conditions of Theorem \ref{T:Ve}, and let $\alpha$ be as in the statement of Theorem \ref{T:Ve}. If  $X=\sum_{i=1}^n a_i X_i$,
then
$X$  has maximal nilpotency if and only if $a_1\not=0$ and $\alpha a_2\not =-a_1$.
\end{lemma}

\begin{proof} Since $n\geq 5$, it is clear from the relations of Theorem \ref{T:Ve} that if  $X$  has maximal nilpotency, then $a_1\not=0$. So, without loss of generality, we will assume that $a_1=1$.

If $\alpha=0$, the relations of Theorem \ref{T:Ve} give  $[X_i,X_j]\in\g_{i+j}$ for all $i,j$. Hence $[X,X_j]=X_{j+1}\pmod{\g_{j+2}}$,
and in particular,
$[X,\g_j]\subseteq\g_{j+1}$, for all $j>2$.
It follows that  $\ad^{n-2}(X)(X_2)=X_n$, and so $X$ has maximal nilpotency.

Now suppose that $\alpha\not=0$. In this case, $[X_i,X_j]\in\g_{i+j}$ whenever $i+j\not=n+1$, and $[X_i,X_{n-i+1}]=(-1)^i \alpha X_n$,  for all $2\leq i\leq n-1$.
Note that for all $2\leq j\leq n-2$,
 \[
[X,X_j]-X_{j+1}=[X-X_1,X_j]=\left[\sum_{i=2}^n a_i X_i,X_j\right]\in \g_{j+2},
\]
and so $[X,X_j]=X_{j+1} \pmod{\g_{j+2}}$,
and in particular,
$[X,\g_j]\subseteq\g_{j+1}$.
It follows that  $\ad^{n-3}(X)(X_2)=X_{n-1}\pmod{\g_{n}}$.
Thus
\[
\ad^{n-2}(X)(X_2)=[X,X_{n-1}]=[X_1+a_2X_2,X_{n-1}]=(1+\alpha a_2)X_n.
\]
Hence, if $1+\alpha a_2\not=0$, then $X$  has maximal nilpotency.

Conversely, if $\ad^{n-2}(X)(Y)\not=0$ for some element $Y$, then  subtracting a multiple of $X$ and rescaling $Y$ if necessary, we may assume that $Y$ has the form $Y=X_2+\sum_{i=3}^n b_i X_i$ for some constants $b_3,\dots,b_n$. But $\ad^{n-2}(X)(X_i)=0$ for $i>2$ and $\ad^{n-2}(X)(X_2)\not=0$ only when $\alpha a_2\not =-1$. Thus $\ad^{n-2}(X)(Y)\not=0$ only when $\alpha a_2\not =-1$.
\end{proof}

\begin{remark}\label{R:basis}
Suppose that $\g$ is a filiform nilpotent metric Lie algebra  having a basis $B=\{X_1,\dots,X_n\}$  satisfying the conditions of Theorem \ref{T:Ve}, and let $\alpha$ be as in the statement of Theorem \ref{T:Ve}. Apply  the Gram-Schmidt orthonormalisation procedure to $B$, starting with the elements of largest $B$-degree, so as to obtain an orthonormal basis $E=\{E_1,\dots,E_n\}$, where for each $i$ one has $\deg_B(E_i)=i$. Observe that if $\g$ is regular, then $E_1$ has maximal nilpotency. If $\g$ is irregular, then by the previous lemma, $E_1$ has maximal nilpotency unless $E_1$ is a multiple of $X_1+\sum_{i=2}^n a_i X_i$, where $1+\alpha a_2=0$.  However, in all cases, one has $\deg_E([E_i,E_j])\geq i+j$, for all $i,j$ with  $i+j\not= n+1$, and  $\deg_E([E_i,E_j])\geq n$, for all $i,j$ with  $i+j= n+1$. In particular,  $\deg_E([E_1,E_i])=i+1$, for all $2\leq i\leq n-2$.

Now let $\h$ be a totally geodesic subalgebra of $\g$.

\end{remark}

\begin{lemma}\label{L:mn1}
Suppose that $\g$ is a filiform nilpotent metric Lie algebra and that $\h$ is a proper totally geodesic subalgebra. If $\dim(\h)\geq 2$, then  $\h$ does not contain an element of maximal nilpotency.
\end{lemma}

\begin{proof}
By Proposition \ref{P:cod1}, filiform nilpotent metric Lie algebras do not possess totally geodesic subalgebras of codimension one. So, as $\dim(\h)\geq 2$, we have $\dim(\g)\geq 4$. If $\dim(\g)= 4$, then as $\g$ is filiform, $\g$ has a derived algebra of dimension two, and we see explicitly from Proposition \ref{P:4D1Dtga}(b) that $\h$ has no element of maximal nilpotency. So we may assume that $\dim(\g)\geq 5$.

Let  $B=\{X_1,\dots,X_n\}$ be a basis satisfying the relations given in Theorem \ref{T:Ve}, and let $\alpha$ be as in the statement of Theorem \ref{T:Ve}.   Suppose that $\h$ contains an element $Y$ of maximal nilpotency; by Lemma \ref{L:Maxnilp} we may assume that $Y$ has the form $Y=X_1+\sum_{i=2}^n a_i X_i$,
where $\alpha a_2\not =-1$. In particular,  from the proof of Lemma \ref{L:Maxnilp}, $[Y,X_i]=X_{i+1} \pmod{\g_{i+2}}$ for all $2\leq i\leq n-2$ and $[Y,X_{n-1}]=(1+\alpha a_2)X_{n}$.
Let  $E=\{E_1,\dots,E_n\}$ be the orthonormal  basis for $\g$ obtained from $B$ as described in Remark \ref{R:basis}. Observe that for all $2\leq i\leq n-1$, we have $[Y,E_i]=\sum_{j\geq i+1} b_{i,j} E_j$, where $b_{i,j}\in\R$ for each $j$ and $b_{i,(i+1)}\not=0$.
Let $k$ be the smallest integer greater than one for which there is $Y_k\in \h$ with $\deg_B(Y_k)=k$; note that $k$ exists as $\dim(\h)\geq 2$. As $\h$ is a subalgebra, $Y_{j+k}=\ad^j(Y)(Y_k)\in \h$ for all $j\geq 1$, and moreover, $\deg_B (Y_{k+j})=\deg_B (E_{k+j})$ for all $0\leq j\leq n-k$, as $Y$ has maximal nilpotency. Hence
\[
\h= \Span(Y,Y_k,Y_{k+1},\dots,Y_{n})=  \Span(Y,E_k,E_{k+1},\dots,E_n),
\]
where the equality on the left follows from the definition of $k$. Note that $k>2$ since otherwise $\h=\g$. Let
 \[
Y'=Y-\sum_{i=k}^n \langle Y,E_i\rangle E_i.
 \]
Notice that $Y'$ has maximal nilpotency. Moroever,   $Y'\in\h$, and $Y'$ is orthogonal to $\Span(E_k,E_{k+1},\dots,E_n)$. By Lemma \ref{L:Maxnilp} and its proof, for all $2\leq i\leq n-1$, we have $[Y',X_i]=X_{i+1}\pmod{\g_{i+2}}$  and  hence $[Y',E_i]=\sum_{j=i+1} ^nb'_{i,j} E_j$ where $b'_{i,j}\in\R$ for each $j$ and $b'_{i,(i+1)}\not=0$.   Consider the element
\[
Z=E_{k-1}-\frac{\langle E_{k-1},Y'\rangle}{\Vert Y'\Vert^2} Y'\in\hp.
\]
We have $ [Y',Z]=[Y',E_{k-1}]=\sum_{j=k}^n b'_{(k-1),j} E_j$ and so $\langle [Y',Z],E_k\rangle=b'_{(k-1),k}\not=0$. However, $\langle Y',[Z,E_k]\rangle=0$ since $[Z,E_k]$ belongs to $\Span(E_{k+1},\dots,E_n) $. Hence
\[
\langle [Z,Y'],E_k\rangle+\langle Y',[Z,E_k]\rangle =-b'_{(k-1),k}\not=0,
\]
contradicting Lemma \ref{L:basic}.
\end{proof}

\begin{remark}\label{R:reg}
If $\g$ is a regular filiform nilpotent metric Lie algebra, then by Remark \ref{R:basis}, $\g$ has an orthonormal basis $E=\{E_1,\dots,E_n\}$ for which $\deg_E([E_i,E_j])\geq i+j$, for all $i,j$, and $\deg_E([E_1,E_i])=i+1$, for all $2\leq i\leq n-1$.  In particular, from Lemma \ref {L:Maxnilp}, the elements of the form $E_1+\sum_{i=2}^n a_iE_i$ all have maximal nilpotency. So if $\h$ is a totally geodesic subalgebra of $\g$, then by Lemma \ref{L:mn1}, $E_1\in \hp $.
Note that, by construction, $\Span(E_k, \dots, E_n)=\Span(X_k, \dots, X_n)=\g_k$, so $\deg_E=\deg_B$.
\end{remark}

\begin{lemma}\label{L:mn2}
Suppose that $\g$ is a filiform nilpotent metric Lie algebra and that $\h$ is a totally geodesic subalgebra for which $\hp$ is $\h$-invariant. If $\hp$ contains an element of maximal nilpotency, then $\dim (\h)\leq (\dim(\g))/2$.
\end{lemma}

\begin{proof}
Suppose that $X\in \hp$ has maximal nilpotency. Consider the map $f:\h\to \hp$ defined by $f(Y)= \ad(X)(Y)$, which is well defined because $\hp$ is $\h$-invariant. We have
\[
\dim( \h) =\dim(\ker f)+\dim (\im f) .
\]
Note that $X \notin \im f$, as $\ad(Y)$ is nilpotent. Thus $f$ is not surjective and so $\dim (\im f)\leq \dim (\hp)-1=\dim(\g)-\dim (\h) -1$. Thus
\[
2\dim (\h) \leq \dim(\ker f)+\dim(\g)-1.
\]
Notice also that as $X$  has maximal nilpotency, $\dim( \ker f)\leq 1$. Hence, from above,
$2\dim( \h )\leq \dim(\g)$, as required.
\end{proof}

\begin{proof}[{Proof of Theorem \ref{T:fili}}]
Let  $B=\{X_1,\dots,X_n\}$ be a basis satisfying the relations given in Theorem \ref{T:Ve}, and let $\alpha$ be as in the statement of Theorem \ref{T:Ve}.   By Lemma \ref{L:mn1}, $\h$ has no element of  maximal nilpotency, while by Lemma \ref{L:mn2}, if $\hp$ has  maximal nilpotency, then $\dim( \h)\leq( \dim(\g))/2$. So we may assume that neither $\h$ nor $\hp$ contains an element of  maximal nilpotency. In particular, by Remark \ref{R:reg}, $\g$ is irregular and hence by Remark \ref{R:dim6}, $\dim(\g)\geq 6$ and $\dim(\g)$ is even.

We have $[\g,\g]=\Span(X_3,\dots,X_n)$. Let $p:\g\to \g/[\g,\g]$ denote the natural quotient map. Clearly, $\g/[\g,\g]$  has dimension two and is generated by $p(X_1)$ and $p(X_2)$. By Proposition \ref{P:nil}, $p(\h)\not=\g/[\g,\g]$. If $p(\h)=0$, then there would exist $X\in\hp$ with $p(X)=p(X_1)$, but then by Lemma \ref{L:Maxnilp},  $X$ would have maximal nilpotency, contrary to our assumption. So $p(\h)$ has dimension 1.
Let $Z_2\in\h$ such that $p(Z_2)$ spans $p(\h)$, and let $Z_1\in\hp$ such that $p(Z_1),p(Z_2)$ span $\g/[\g,\g]$. Since neither $Z_1$ nor $Z_2$ has  maximal nilpotency, by rescaling if necessary, we may write, by Lemma \ref{L:Maxnilp},  $Z_1=X_1-\frac1{\alpha}X_2+\sum_{i=3}^n c_i X_i$ and  $Z_2=X_2+\sum_{i=3}^n b_i X_i$, or alternately, $Z_2=X_1-\frac1{\alpha}X_2+\sum_{i=3}^n c_i X_i$ and  $Z_1=X_2+\sum_{i=3}^n b_i X_i$.
Notice that in either case, since $\hp$ is $\h$-invariant, $[Z_1,Z_2]\in \hp$, and $[Z_1,Z_2]=\pm X_3\pmod{\g_{4}}$. If $Z_2=X_1-\frac1{\alpha}X_2+\sum_{i=3}^n c_i X_i$, then by repeatedly bracketing $Z_1$ with $Z_2$, we obtain that $\hp$ contains linearly independent vectors of the form $X_i\pmod{\g_{i+1}}$ for $i=2,\dots,n-1$. So $\h$ has dimension at most 2, and we are done.
So we may assume that $Z_1=X_1-\frac1{\alpha}X_2+\sum_{i=3}^n c_i X_i$.  The idea is to adapt the proof of Lemma \ref{L:mn2}. Notice that while $Z_1$ does not have maximal nilpotency, one does have $\dim(\ker \ad(Z_1))=3$; see the proof of Lemma \ref{L:Maxnilp}. Consider the map $f:\h\to \hp$ defined by $f(Y)=\ad(Z_1)(Y)$. Note that  $Z_1\not\in \im f$ since  $Z_1$ does not belong to $[\g,\g]$ while the image of $f$ does. Thus $f$ is not surjective and so $\dim (\im f)\leq \dim (\hp)-1=\dim(\g)-\dim (\h) -1$. Thus
\[
2\dim (\h) \leq \dim(\ker f)+\dim(\g)-1.
\]
Notice also that as $\dim(\ker \ad(Z_1))=3$, and $Z_1\in \ker \ad(Z_1)$, we have $\dim (\ker f)\leq 2$. Hence
$2\dim (\h) \leq \dim(\g)+1$. Thus, as $\dim(\g)$ is even, $\dim (\h) \leq (\dim(\g))/2$, as required.
This completes the proof.
\end{proof}

\begin{proof}[{Proof of Theorem \ref{T:conv}}]
Up to Lie algebra isomorphism, there is only one filiform nilpotent Lie algebra in dimensions $3$ and $4$, namely the standard filiform algebras $L_3,L_4$ respectively. So we may assume that $\dim(\g)\geq 5$.

Let  $B=\{X_1,\dots,X_n\}$ be a basis satisfying the relations given in Theorem \ref{T:Ve}, and let $\alpha$ be as in the statement of Theorem \ref{T:Ve}.   Let  $E=\{E_1,\dots,E_n\}$ be the orthonormal  basis for $\g$ obtained from $B$ as described in Remark \ref{R:basis}. Choose a basis $\{Y_1,\dots,Y_{n-2}\}$ for $\h$ such that $\deg_E(Y_i)<\deg_E(Y_j)$ for all $1\leq i<j\leq n-2$.
Repeating the arguments of the proof of Lemma \ref{L:cod2}, we obtain a codimension one  ideal $\m$ of $\g$ that contains $\h$ as a codimension one subalgebra.  Let  $Z_1\in \g$ be a nonzero vector orthogonal to $\m$, and let  $Z_2\in \m$ be a nonzero vector orthogonal to $\h$.  Note that $\h$ is a codimension one totally geodesic subalgebra of $\m$, so by Proposition \ref{P:cod1}, $\m$ is the direct sum of Lie ideals, $\m=\h\oplus \Span(Z_2)$. In particular, $Z_2$ commutes with all the elements of $\h$, and so $\deg_E(Z_2)\geq 2$.

First  suppose that $\deg_E(Y_1)=1$. By Lemma \ref{L:mn1}, $\h$ has no elements of maximal nilpotency. So by Lemma \ref{L:Maxnilp},  $\g$ is irregular  and hence has even dimension $n\geq 6$, and $Y_1$ is a multiple of an element of the form $X_1+\sum_{i=2}^n a_i X_i$, where  $\alpha a_2=-1$. In particular, $Y_1$ commutes with $X_{n-1}$ and $X_n$, but $[Y_1,X_i]\not=0$ for $2\leq i\leq n-2$. Since $Z_2$ commutes with all the elements of $\h$, we have $[Z_2,Y_1]=0$. Thus $\deg_E(Z_2)\geq n-1$. Notice that $\deg_E(Y_2)>2$. Indeed, otherwise, $\h$ would project surjectively onto $\g/ [\g,\g]$, contradicting Proposition \ref{P:nil}. Thus  $\deg_E(Z_1)\leq 2$.   Hence,  since $\h$ has codimension 2, and since $3<n-1$ as $n\geq 6$, we have $\deg_E( Y_2)=3$. Repeatedly applying $\ad(Y_1)$, we obtain $\deg_E( Y_i)=i+1$ for all $2\leq i\leq n-2$. So $\h$ has no element of $E$-degree $n$. However, by Theorem \ref{T:Ve},  since $3<n-2$ as $n\geq 6$, $[X_3,X_{n-2}]=-\alpha X_n$, and hence $\deg_E([Y_2,Y_{n-3}])=n$. Thus $E_n\in \h$, which is a contradiction. Thus $\deg_E(Y_1)=1$ is impossible.

Now suppose that $\deg_E(Y_1)>2$. In this case we have $\deg_E(Z_1)=1$ and $\deg_E(Z_2)= 2$. Thus, since $\h$ has codimension 2, we have $\deg_E(Y_i)=i+2$, for all $1\leq i\leq n-2$. Consequently $E_i\in \h$ for all $3\leq i\leq n$. In particular, $E_3,E_4\in\h$. But then, as $n>4$,
\[
\langle[Z_1,E_3],E_4\rangle+\langle[Z_1,E_4],E_3\rangle=\langle[Z_1,E_3],E_4\rangle\not=0,
\]
in contradiction with Lemma \ref{L:basic}. So the case  $\deg_E(Y_1)>2$ is impossible.

It remains to treat the case where $\deg_E(Y_1)=2$. Here $\deg_E(Z_1)=1$.
Since $\h$ has codimension 2, there is precisely one integer $k$, with $3\leq k\leq n$ such that $\h$ does not have an element of $E$-degree $k$. Then, by change of basis if necessary, we may assume that the vectors $Y_i$ have the following form
\[
Y_i=
\begin{cases}
E_{i+1}+b_i E_k&:\ \text{for}\ 1\leq i\leq k-2\\
E_{i+2}&: \ \text{for}\  i\geq k-1,
\end{cases}
\]
where $b_i\in\R$ for $1\leq i\leq k-2$. Hence $\hp=\Span(E_1,E_k-\sum_{i=1}^{k-2} b_{i} E_{i+1})$.
First suppose that $\deg_E(Z_2)>3$.  Then we may set $Y_1=E_2$ and $Y_2=E_3$. Hence
\[
\langle [E_1,Y_1],Y_2\rangle+\langle [E_1,Y_2],Y_1\rangle=\langle [E_1,E_2],E_{3} ]\rangle\not=0,
\]
contradicting Lemma \ref{L:basic}. Thus $\deg_E(Z_2)>3$ is impossible.

Now suppose that $\deg_E(Z_2)=2$. So $\hp$ would project surjectively onto $\g/ [\g,\g]$.  By Proposition \ref{P:nil}, $\g$ has no proper subalgebra that contains  $\hp$. Hence by Lemma \ref{L:cod2}(b),
 $\h\subseteq \z(\h)$; that is, $\h$ is abelian. Thus, since $\m=\h\oplus \Span(Z_2)$, the ideal $\m$ is also abelian. Hence $\g$ is a standard filiform Lie algebra.

Finally, suppose that $\deg_E(Z_2)=3$. Let $\a$ denote the smallest subalgebra of $\g$ that contains $\hp$. Since $\deg_E(Z_2)=3$ and $\deg_E(Z_1)=1$, the algebra $\a$ contains elements of $E$-degree $i$ for all $3\leq i\leq n-1$ (and for all $3\leq i\leq n$  if $\ad(Z_1)$ has maximal nilpotency). In particular, $\Span(\a,E_n)$ has dimension $ n-1$. Hence $\h\cap\Span(\a,E_n)$ has codimension 3 in $\g$, and codimension one in $\h$. By Lemma \ref{L:cod2}(b),
 $\a\cap \h \subseteq \z(\h)$. So, since $E_n\in\z(\g)$, we have $\h\cap\Span(\a,E_n)\subseteq \z(\h)$. Since no Lie algebra of dimension greater than two has a centre of codimension one, we conclude that  $\h$ is abelian, and thus $\g$ is standard.
\end{proof}

\begin{remark}
In the proof of Theorem   \ref{T:conv}, in the notation used in the proof, we showed that the only possibilities are $\deg_E(Z_2)=2$ or $\deg_E(Z_2)=3$, and in these cases $\g$ is standard. In particular, $Z_1=E_1$ has maximal nilpotency. Moreover, if $k<n-1$, then $E_{n-1},E_n\in \h$, and so
\[
\langle [E_1,E_{n-1}],E_{n}\rangle+\langle [E_1,E_{n}],E_{n-1}\rangle=\langle [E_1,E_{n-1}],E_{n}\rangle\not=0,
\]
contradicting Lemma \ref{L:basic}. Thus  $k=n-1$ or $k=n$. First suppose that $k=n-1$. Then
\[
Y_i=
\begin{cases}
E_{i+1}+b_i E_{n-1}&:\ \text{for}\ 1\leq i\leq n-3\\
E_{n}&: \ \text{for}\  i=n-2,
\end{cases}
\]
Let $[E_1,E_i]=\sum_{j=i+1}^n c_{ij}E_j$, for $i=2,\dots,n-1$. For all $1\leq i\leq n-3$, Lemma \ref{L:basic} gives $\langle [E_1,Y_{i}],E_{n}\rangle=0$ and thus
\begin{equation}\label{E:cb1}
c_{(i+1)n}+b_ic_{(n-1)n}=0.
\end{equation}
Similarly, for all $1\leq i<j\leq n-3$, Lemma \ref{L:basic} gives
\[
\langle [E_1,Y_{i}],Y_{j}\rangle+\langle [E_1,Y_{j}],Y_{i}\rangle=0,
\]
 and thus
\begin{equation}\label{E:cb2}
c_{(i+1)(j+1)}+b_jc_{(i+1)(n-1)}+b_ic_{(j+1)(n-1)}=0.
\end{equation}
Notice that (\ref{E:cb1}) determines the $b_i$ in terms of the commutator coefficients $c_{ij}$, while (\ref{E:cb2}) gives quadratic conditions on the $c_{ij}$. One solution to this system of equations is given below in the proof of Theorem \ref{T:cd2f}. The case where $k=n$ can be treated in a similar fashion.
\end{remark}

\section{Examples}\label{S:ex}

\begin{proof}[{Proof of Proposition \ref{P:4D1Dtga}}]  Choose an orthonormal basis $\{X_1,X_2,X_3,X_4\}$ for $\g$ such that $[\g,\g]=\Span(X_3,X_4)$, and $X_4 \in \z(\g)$. As $\g$ is nilpotent,
    neither of the two vectors $[X_1,X_3], [X_2,X_3] \in [\g,\g]=\Span( X_3,X_4)$ can have a nonzero $X_3$-component, so we can choose orthonormal
    vectors     $X_1, X_2 \in [\g,\g]^\perp$ in such a way that $[X_1,X_3]= \gamma X_4$, $[X_2,X_3]=0$. Then $[X_1,X_2] =\alpha X_3+\beta X_4$, for some
    $\alpha,\beta\in \R$, so the relations have the required form, and moreover, $\alpha$ and $\gamma$ must be nonzero (as
    otherwise $\dim( [\g,\g]) < 2$). Changing the signs of $X_3$ and $X_4$, if necessary, we get $\alpha,\gamma>0$.

 (a)    Clearly, $[\g,\g]=\Span( X_3,X_4)$ and $\z(\g)=\Span( X_4)$. So by Lemma  \ref{L:center}, $X_4$ and the nonzero elements of $\Span(X_1,X_2)$
    are all geodesics.
    Let $X\in\g$ be nonzero, and write $X=\sum_{i=1}^4a_iX_i$. Then $X$ is a geodesic if and only if $\nabla_X X=0$; that is, if
    $\langle [X,X_i],X\rangle=0$, for all $i=1,\dots, 4$ (see Remark~\ref{R:geo}). First suppose that $a_1\not=0$. Note that
    $\langle [X,X_3],X\rangle=\gamma a_1a_4$, which is zero only if $a_4=0$. Then $\langle [X,X_2],X\rangle=\alpha a_1a_3$, which is zero only if
    $a_3=0$. So if $a_1\not=0$, we have $ X\in\Span( X_1,X_2)$.
    Now suppose that $a_1=0$. The conditions $\langle [X,X_i],X\rangle=0$ are trivial for $i=2,3,4$. The condition $\langle [X,X_1],X\rangle=0$ is
    equivalent to $a_2(\alpha a_3 +\beta a_4) +\gamma a_3a_4=0$, as required.

(b) The subalgebra $\h$ is nilpotent, hence abelian, so every nonzero vector in $\h$ is a geodesic in $\h$, and hence in $\g$, by \eqref{L:TotGeod-cond}
of Lemma~\ref{L:basic}. It follows that $\h$ must be a two-dimensional abelian subspace in the semi-algebraic geodesic variety from (b). If
$\beta \ne 0$, the only two-dimensional subspace of that variety is $\Span( X_1,X_2)$, which is not abelian (and not a subalgebra). If $\beta=0$,
the geodesic variety is the union of three subspaces: $\Span( X_1,X_2), \; \Span(X_2,X_4)$, and $\h=\Span(\gamma X_2-\alpha X_4,X_3)$.
The former one is again not a subalgebra, while the other two are totally geodesic subalgebras, as directly follows from \eqref{L:TotGeod-cond}.
\end{proof}

\begin{remark}\label{R:C=L}
Let $C=\{c_2,\dots,c_{n-1}\}$ be an ordered list of nonzero reals, and let $L_C$ denote the  nilpotent filiform metric Lie algebra $\g$, with orthonormal basis $\{X_1,\dots,X_n\}$, whose only nonzero relations are $[X_1,X_i]=c_iX_{i+1}$, for $i=2,\dots,n-1$. Clearly, as a Lie algebra, $L_C$ is isomorphic to the standard filiform algebra $L_n$. We claim that if $L_C$ possesses a totally geodesic  subalgebra of dimension $k$, then so too does the metric Lie algebra $L_n$. Indeed, we will use the notation $\{X_1,\dots,X_n\}$ for the bases in both $L_C$ and $L_n$, and  we denote the inner products in the two algebras  by the same notation $\langle \cdot,\cdot\rangle$; that is, we think of the inner product spaces as being identical, and only the Lie bracket as being different. We will denote the bracket in $L_n$ by $[ \cdot,\cdot]$, and the bracket in $L_C$ by $[ \cdot,\cdot]_C$. Define constants $f_2,\dots,f_n$ recursively by posing $f_2=1$ and $f_if_{i+1}=c_i$, for $i=2,\dots,n-1$, and consider the vector space isomorphism $\phi: L_C \to L_n$ defined by $X_1\mapsto X_1, X_i\mapsto f_iX_i$ for $i\geq2$. In general, $\phi$ is not a Lie algebra homomorphism, but we claim that it preserves totally geodesic subalgebras, and in particular, the map $\phi$ is ``projective'', in the sense that it preserves geodesics. Indeed, suppose that $L_C$ has a totally geodesic  subalgebra $\h$. Note that for both $L_n$ and $L_C$, we have $X_1\in \hp$, by Remark \ref{R:reg}, so $\h$ has a basis
$\{Y_1,\dots,Y_k\}$, where for $i=1,\dots,k$,
\[
Y_i=\sum_{j=2}^n a_{i,j} X_j,
\]
for some constants $a_{i,j}\in\mathbb R$. The image $\phi(\h)$ is contained in the abelian ideal\break $\Span(X_2, X_3,\dots , X_n)$, so it is a subalgebra; it has basis $\{Z_1,\dots,Z_k\}$, where
\[
Z_i=\phi(Y_i)=\sum_{j=2}^n a_{i,j}f_j X_j,
\]
 for $i=1,\dots,k$. Since $\h$ is totally geodesic, we have
 $ \langle [X_1,Y_i]_C,Y_l\rangle+\langle [X_1,Y_l]_C,Y_i\rangle=0$,
 for all $1\leq i,l\leq k$. That is, setting $a_{i,(n+1)}=0$, we have
  \[
 \sum_{j=2}^n c_j a_{i,j} a_{l,(j+1)} +c_j a_{l,j} a_{i,(j+1)}=0,
  \]
or equivalently
 \[
 \sum_{j=2}^n a_{i,j} f_j a_{l,(j+1)} f_{j+1} +a_{l,j} f_j a_{i,(j+1)}f_{j+1} =0.
  \]
 Hence
$ \langle [X_1,Z_i],Z_l\rangle+\langle [X_1,Z_l],Z_i\rangle=0$,
for all $1\leq i,l\leq k$, and so the algebra $\phi(\h)$ is totally geodesic  in $L_n$, as claimed.
\end{remark}

\begin{proof}[{Proof of Theorem \ref{T:cd2f}}]
We define an orthonormal basis $\{E_1,\dots,E_n\}$ as follows: we set $E_1=X_1, E_{n}=X_n$ and for each $i\in \{2,3,\dots,n-1\}$, we pose
\[
E_i=
\sum_{j=0} ^{\lfloor\frac{n-1-i}{2}\rfloor} \begin{pmatrix}n- 1-i-j\\
j
\end{pmatrix}X_{i+2j}.
\]
Notice that the basis has been chosen so that
\[
[E_1,E_i]=E_{i+1}+E_{i+3}+E_{i+5}+\dots,
\]
for all $i\geq 2$. Indeed, for all $2\leq i\leq n-3$,
\begin{align*}
E_i-E_{i+2}
&=\sum_{j=0} ^{\lfloor\frac{n-1-i}{2}\rfloor} \begin{pmatrix}n- 1-i-j\\
j
\end{pmatrix}X_{i+2j}- \sum_{j=0} ^{\lfloor\frac{n-3-i}{2}\rfloor} \begin{pmatrix}n- 3-i-j\\
j
\end{pmatrix}X_{i+2+2j}\\
&=\sum_{j=0} ^{\lfloor\frac{n-1-i}{2}\rfloor} \begin{pmatrix}n- 1-i-j\\
j
\end{pmatrix}X_{i+2j}- \sum_{j=1} ^{\lfloor\frac{n-1-i}{2}\rfloor} \begin{pmatrix}n- 2-i-j\\
j-1
\end{pmatrix}X_{i+2j}\\
&=\sum_{j=0} ^{\lfloor\frac{n-1-i}{2}\rfloor} \left[\begin{pmatrix}n- 1-i-j\\
j
\end{pmatrix}- \begin{pmatrix}n- 2-i-j\\
j-1
\end{pmatrix}\right]X_{i+2j}\\
&=\sum_{j=0} ^{\lfloor\frac{n-1-i}{2}\rfloor} \begin{pmatrix}n- 2-i-j\\
j
\end{pmatrix}X_{i+2j},
\end{align*}
so
\begin{align*}
[E_1,E_i-E_{i+2}]
&=\sum_{j=0} ^{\lfloor\frac{n-1-i}{2}\rfloor} \begin{pmatrix}n- 2-i-j\\
j
\end{pmatrix}X_{i+1+2j}\\
&=\sum_{j=0} ^{\lfloor\frac{n-1-(i+1)}{2}\rfloor} \begin{pmatrix}n- 1-(i+1)-j\\
j
\end{pmatrix}X_{i+1+2j},
\end{align*}
since if $\lfloor\frac{n-1-(i+1)}{2}\rfloor<\lfloor\frac{n-1-i}{2}\rfloor$, then $\lfloor\frac{n-1-i}{2}\rfloor=\frac{n-1-i}{2}$, and so for $j=\lfloor\frac{n-1-i}{2}\rfloor$, we have
\[
\begin{pmatrix}n- 2-i-j\\
j
\end{pmatrix}=\begin{pmatrix}(n-3-i)/2\\
(n-1-i)/2
\end{pmatrix}=0.
\]
Thus
$[E_1,E_i-E_{i+2}]=E_{i+1}$ for all $2\leq i\leq n-3$. By definition, $E_{n}=X_{n},E_{n-1}=X_{n-1}$ and so $[E_1,E_n]=0$ and $[E_1,E_{n-1}]=E_n$. Hence, by induction, $[E_1,E_i]=E_{i+1}+E_{i+3}+E_{i+5}+\dots$, as claimed.

Let $\h$ denote the codimension two subalgebra   $\Span(Y_2,Y_3,\dots,Y_{n-2},Y_{n})$, where  for $i\in\{2,3,\dots,n-2,n\}$,
\[
Y_{i}=
\begin{cases}
E_{i} &:\ \text{if}\ n-i\ \text{is even},\\
E_{i} -E_{n-1} &:\ \text{otherwise}.
\end{cases}
\]
Notice that the orthogonal complement $\hp$ to $\h$ is spanned by the vectors
\[
Z_1:=E_1\ \text{and}\ Z_2:=\sum_{\substack{1\leq j\leq n-2\\ j\ \text{odd}}} E_{n-j}.
\]
Clearly, $Z_2$ commutes with each of the vectors $Y_i$, so to show that $\h$ is totally geodesic, we must show that
\begin{equation}\label{E:cond}
\langle [Z_1,Y_i],Y_j\rangle + \langle [Z_1,Y_j],Y_i\rangle =0,
\end{equation}
for all $i,j\in\{2,3,\dots,n-2,n\}$. If $n-i,n-j$ are both even, then
\[
\langle [Z_1,Y_i],Y_j\rangle = \langle [E_1,E_i],E_j\rangle =\langle E_{i+1}+E_{i+3}+\dots+E_{n-1} ,E_j\rangle =0,
\]
and similarly $ \langle [Z_1,Y_j],Y_i\rangle=0$; so (\ref{E:cond}) holds in this case.
Notice that if $n-i$ is odd, then $[E_1,E_i] = E_{i+1}+E_{i+3}+\dots +E_{n}$. So if $n-i,n-j$ are both odd, then
\begin{align*}
\langle [Z_1,Y_i],Y_j\rangle &= \langle [E_1,E_i-E_{n-1}],E_j-E_{n-1}\rangle \\
&=\langle E_{i+1}+E_{i+3}+\dots +E_{n-2},E_j-E_{n-1}\rangle =0,
\end{align*}
\and similarly $ \langle [Z_1,Y_j],Y_i\rangle=0$; so (\ref{E:cond}) again holds.
It remains to consider the case where $n-i$ is odd and $n-j$ is even.  If $i>j$, then
\[
\langle [Z_1,Y_i],Y_j\rangle  = \langle [E_1,E_i-E_{n-1}],E_j\rangle
=\langle E_{i+1}+E_{i+3}+\dots +E_{n-2},E_j\rangle =0,
\]
and
\begin{align*}
\langle [Z_1,Y_j],Y_i\rangle  &= \langle [E_1,E_j],E_i-E_{n-1}\rangle\\
&=\langle E_{j+1}+E_{j+3}+\dots +E_{n-1},E_i-E_{n-1}\rangle\\
&=\langle E_{i}+E_{n-1},E_i-E_{n-1}\rangle =0,
\end{align*}
so (\ref{E:cond}) again holds.  Finally, if $i<j$,
\begin{align*}
\langle [Z_1,Y_i],Y_j\rangle&+\langle [Z_1,Y_j],Y_i\rangle  = \langle [E_1,E_i-E_{n-1}],E_j\rangle +\langle [E_1,E_j],E_i-E_{n-1}\rangle\\
&=\langle E_{i+1}+E_{i+3}+\dots +E_{n-2},E_j\rangle \\
&\quad  +\langle E_{j+1}+E_{j+3}+\dots +E_{n-1},E_i-E_{n-1}\rangle \\
&=1-1=0.
\end{align*}
Hence (\ref{E:cond}) holds in all cases.
\end{proof}

\begin{proof}[{Proof of Theorem \ref{T:dim6}}]
Fix an inner product on $\g$ and suppose that $\h$ is a totally geodesic proper subalgebra of $\g$. Notice that $\g$ is not isomorphic to $L_6$ since its derived algebra is not abelian. So by Proposition \ref{P:cod1} and Theorem \ref{T:conv}, $\h$ has codimension at least three. Suppose that $\h$ has codimension three. For each $k=1,\dots,6$, let $\g_k=\Span( X_k,\dots,X_6)$.   The basis $\{X_1,\dots,X_6\}$ satisfies the condition of  Theorem \ref{T:Ve} with $\alpha =0$; that is, $\g$ is regular.
 Let  $E=\{E_1,\dots,E_6\}$ be the orthonormal  basis for $\g$ obtained from the basis $B=\{X_1,\dots,X_6\}$ as described in Remark \ref{R:basis}.
Notice that by construction, $[E_1,E_i]\in\g_{i+1}=\Span(X_{i+1},\dots,X_6)= \Span(E_{i+1},\dots,E_6)$, and $\langle[E_1,E_i],E_{i+1}\rangle  \not=0$, for each $i=2,\dots,5$, and $\langle[E_2,E_3],E_6\rangle  \not=0$. Moreover, $[E_i,E_j]\in\g_{i+j}$ for $i,j$, and $[E_2,E_4]=0$ and $[E_i,E_j]=0$ for $i+j>6$. By Remark \ref{R:reg}, we have $E_1\in\hp$. Let $\hp=\Span(E_1,Z_1,Z_2)$, where $2\leq \deg_E(Z_1)<\deg_E(Z_2)$.

We will repeatedly use the following fact: it is impossible that $E_i,E_{i+1}\in\h$, for any $i=2,\dots,5$.  Indeed, otherwise
\[
\langle[E_1,E_i],E_{i+1}\rangle+\langle[E_1,E_{i+1}],E_i\rangle=\langle[E_1,E_i],E_{i+1}\rangle\not=0,
\]
which would contradict Lemma \ref{L:basic}.

First notice that $\deg_E(Z_1)< 4$, since otherwise  $E_2,E_3\in\h$, which is impossible.\break
Suppose that $\deg_E(Z_1)=3$. Then  $E_2\in\h$ and we may take $Z_1=E_3+a_4E_4 +a_5E_5+a_6E_6$, for some $a_4,a_5,a_6\in\R$. Consider an arbitrary nonzero element $Y\in\h$. Then Lemma \ref{L:basic} gives
\begin{equation*}
0=\langle [Z_1,E_2],Y\rangle +\langle[Z_1,Y],E_2\rangle=\langle [Z_1,E_2],Y\rangle =-\langle[E_2,E_3],Y\rangle.
\end{equation*}
As $[E_2,E_3]$ is a nonzero multiply of $E_6$, we have that $E_6$ is orthogonal to $\h$. Hence $E_6\in\h^\perp$.
Thus, without loss of generality, we may take $Z_1=E_3+a_4E_4 +a_5E_5$ and $Z_2=E_6$. Then $a_4E_3-E_4,a_5E_3-E_5\in\h$.
Since $\h$ is a subalgebra and $E_6\in\hp$ we have
\begin{align*}
0&=\langle[E_2,a_4E_3-E_4],E_6\rangle=a_4\langle[E_2,E_3],E_6\rangle\implies a_4=0,\\
0&=\langle[E_2,a_5E_3-E_5],E_6\rangle=a_5\langle[E_2,E_3],E_6\rangle\implies a_5=0.
\end{align*}
But this would give $E_4,E_5\in\h$, which is impossible. So $\deg_E(Z_1)\not=3$.

Finally, suppose that $\deg_E(Z_1)=2$. We may take $Z_1=E_2+a_3E_3+a_4E_4 +a_5E_5+a_6E_6$, for some $a_3,\dots,a_6\in\R$.
Let us first assume that $\deg_E(Z_2)\geq 4$. Then $a_3E_2-E_3\in\h$.  So Lemma \ref{L:basic} gives
\[
0=\langle[E_1,a_3E_2-E_3],a_3E_2-E_3\rangle=-a_3\langle[E_1,E_2],E_3\rangle,
\]
and so $a_3=0$. Thus $E_3\in\h$. Arguing as before, if
 $Y=b_2E_2+b_4E_4 +b_5E_5+b_6E_6\in\h $, then
Lemma \ref{L:basic} gives
\[
0=\langle[Z_1,E_3],Y\rangle+\langle[Z_1,Y],E_3\rangle=\langle[Z_1,E_3],Y\rangle=b_6\langle[E_2,E_3],E_6\rangle,
\]
and so $b_6=0$. Hence $E_6\in\hp$. Thus  $\hp=\Span(E_1,E_2+a_4E_4 +a_5E_5,E_6)$ and $\h=\Span(E_3,a_4E_2-E_4 ,a_5E_2-E_5)$. However $\h$ is a subalgebra and $[E_2,E_3]$ is a nonzero multiple of $E_6$. Hence, as $E_6\not \in\h$, we necessarily have $a_4=a_5=0$. But this gives $E_4,E_5\in\h$, which is impossible. So $\deg_E(Z_2)=3$.
Without loss of generality, we may take
\begin{align*}
Z_1&=E_2+a_4E_4 +a_5E_5+a_6E_6\\
Z_2&=E_3+c_4E_4 +c_5E_5+c_6E_6.
\end{align*}
Consider the orthogonal complement $\g_2$ to $\Span(E_1)$.  Note that $\h\subseteq \g_2$ and $Z_1,Z_2\in\g_2$. It follows that $\h$ is a codimension 2 totally geodesic subalgebra of $\g_2$. Hence by Lemma \ref{L:cod2}(a)(ii), we have $[Z_1,Z_2]\in\h$. Therefore, $E_6\in\h$.
Suppose that $Y=b_2E_2+b_3E_3+b_4E_4+b_5E_5\in\h$, for some $b_2,\dots,b_5\in\R$.
Then Lemma \ref{L:basic} gives
\begin{align*}
0&=\langle[Z_1,Y],E_6\rangle+\langle[Z_1,E_6],Y\rangle=b_3\langle[E_2,E_3],E_6\rangle\implies b_3=0,\ \text{and}\\
0&=\langle[Z_2,Y],E_6\rangle+\langle[Z_2,E_6],Y\rangle=-b_2\langle[E_2,E_3],E_6\rangle\implies b_2=0.
\end{align*}
Hence $E_4,E_5\in\h$, which is impossible. This completes the proof.
\end{proof}

\section*{Appendix}

Mich\`ele Vergne's  classification of filiform Lie algebras is well known to experts, but it is not explicitly stated in \cite{Ve}. Instead, filiform graded Lie algebras are classified in \cite{Ve} and from this, the classification of filiform Lie algebras can be deduced.  It is possibly for this reason that the classification is perhaps not as well known as it merits. For completeness, we adapt the ideas in \cite{Ve}  to provide a direct, elementary proof of Theorem \ref{T:Ve}. The key step is taken directly from \cite[Prop.~5]{Ve}.

\begin{proof}[{Proof of Theorem \ref{T:Ve}}]
Suppose that $\g$ is a filiform Lie algebra of dimension $n$ and that $X_1\in \g$ has $\ad^{n-2}(X_1)\not=0$. Choose a vector $X_2$ with $\ad^{n-2}(X_1)(X_2)\not=0$. Note that as $\g$ is nilpotent, $\ad^{n-1}(X_1)(X_2)=0$. Define $X_{i+2}=\ad^{i}(X_1)(X_2)$, for $i\geq 1$; so $[X_1,X_i]=X_{i+1}$ for all $i\geq 2$. Moreover, the elements $X_1,\dots,X_n$ are linearly independent. Indeed, first notice that  if $X_2,\dots,X_n$ are not linearly independent, then $X_k,\dots,X_n$ are linearly independent for some minimal $k>2$. Thus $X_{k-1}=a_kX_k+\dots+a_nX_n$ for some constants $a_i$. Then applying $\ad(X_1)$ gives
\[
X_k=a_kX_{k+1}+\dots+a_{n-1}X_n,
\]
which is impossible as $X_k,\dots,X_n$ are linearly independent. Thus $X_2,\dots,X_n$ are {li\-nearly} independent. We now show that $X_1$ does not belong to $ \Span( X_2,\dots,X_n)$. Indeed, otherwise  $X_1=a_2X_2+\dots+a_nX_n$ for some constants $a_i$, and thus
\[
0=[X_1,X_1]=[X_1,a_2X_2+\dots+a_nX_n]=a_2X_3+\dots+a_{n-1}X_n,\]
which gives $a_2=\dots=a_{n-1}=0$. Then we would have $X_1=a_nX_n$ for some nonzero $a_n$. But then $\ad(X_{n-1})(X_1)=-[X_1,X_{n-1}]=-X_n=-\frac1{a_n}X_1$, which is impossible as $\ad(X_{n-1})$ is nilpotent, since $\g$ is nilpotent. Thus  $X_1,\dots,X_n$ forms a basis for $\g$, and by construction, $[X_1,X_i]=X_{i+1}$ for all $ i\geq 2$.

\begin{lemma}\label{L:br}
If  $\g$ is filiform nilpotent, then for all bases $\{X_1,\dots,X_n\}$ with $[X_1,X_i]=X_{i+1}$ for all $ i\geq 2$, one has $[X_i,X_j]\in \g_{i+j-1}$, for all $i,j$.\end{lemma}

\def\proofname{Proof}
\begin{proof}
First we will show that $[X_i,X_j]\in \g_{\max\{i,j\}+1}$, for all $i,j$.
This can be seen by induction on $n$. It is true trivially for $n\leq 2$. For $n>2$, let $\{X_1,\dots,X_n\}$ be a basis for $\g$ with the given relations. Notice that the centre $\z(\g)$ of $\g$ is $\Span( X_n)$. Indeed, it is clear from the given relations that $\z(\g)\subseteq \Span( X_n)$, and since $\g$ is nilpotent, $\z(\g)\not=0$; so $\z(\g)=\Span( X_n)$. Let $\g'=\g/\z(\g)$ and for each $i=1,\dots,n-1$, let $X'_i$ denote the image of $X_i$ in $\g'$. Notice that $[X'_1,X'_i]=X'_{i+1}$ for all $ 2\leq i\leq n-2$, and $[X'_1,X'_{n-1}]=0$. So by the inductive hypothesis, $[X'_i,X'_j]\in \g'_{\max\{i,j\}+1}$, for all $i,j\leq n-1$. Hence, as $\z(\g)=\Span( X_n)$, we have $[X_i,X_j]\in\g_{\max\{i,j\}+1}$, for all $i,j\leq n$, as claimed.

We prove that $[X_i,X_j]\in \g_{i+j-1}$ for all $i,j$. We use induction on $i+j$. The claim is obviously true for $i+j\leq 3$. Suppose $2<i<j$. The Jacobi identity gives
\begin{align*}
[X_i,X_{j}]+[X_{i-1},X_{j+1}]&=[[X_1,X_{i-1}],X_{j}]+[X_{i-1},[X_1,X_{j}]]\\
&=[X_1,[X_{i-1},X_{j}]].
\end{align*}
By the inductive hypothesis, $[X_{i-1},X_{j}]\in\g_{i+j-2}$ and so $[X_1,[X_{i-1},X_{j}]]\in\g_{i+j-1}$. Hence
$[X_i,X_{j}]+[X_{i-1},X_{j+1}]\in\g_{i+j-1}$. It follows that
\[
[X_i,X_{j}]=(-1)^i[X_{2},X_{i+j-2}] \pmod{\g_{i+j-1}},
\]
for all $2\leq i<j$. But from above, $[X_{2},X_{i+j-2}]\in\g_{\max\{2,i+j-2\}+1}=\g_{i+j-1}$. Hence $[X_i,X_{j}]\in \g_{i+j-1}$, as required.
\end{proof}

So far, we have $[X_1,X_j]\in \g_{j+1}$ for all $j$, and from the above lemma, we have $[X_2,X_3]\in\g_{4}$ and so $[X_2,X_3]=b X_{4} \pmod{\g_{5}}$ for some $b\in\R$. Replace $X_2$ by the vector $X_2-bX_1$; by abuse of language we continue to call this $X_2$. Notice that $X_1,\dots,X_n$ continue to form a basis, we still have the relations $[X_1,X_i]=X_{i+1}$ for all $ i\geq 2$ and moreover we still have $[X_i,X_j]\in \g_{\max\{i,j\}+1}$, for all $i,j\leq n$. But now we also have $[X_2,X_3]\in \g_{5}$.

The completion of the proof is by induction on $n$. Notice that if $n\leq 4$, the theorem follows from the properties we have established so far (with $\alpha=0$ in the case $n=4$). So assume that $n>4$ and assume that the properties hold for $X_1,\dots,X_{n-1}$. We first suppose that $n$ is even. As in the proof of Lemma \ref{L:br}, let $\g'=\g/\Span( X_n)$ and for each $i=1,\dots,n-1$, let $X'_i$ denote the image of $X_i$ in $\g'$. By the inductive hypothesis, for all $i+j\leq n$ we have $[X'_i,X'_j]\in \g'_{i+j}$ and so $[X_i,X_j]\in \g_{i+j}$. If $i+j\geq n+2$, then by Lemma \ref{L:br}, $[X_i,X_j]\in \g_{i+j-1}$, and so $[X_i,X_j]\in \g_{i+j}$, since $\g_{i+j}=\g_{i+j-1}=0$. It remains to treat the case where $i+j=n+1$.
We must show that there exists $\alpha\in\R$ with $[X_i,X_{n-i+1}]=(-1)^i \alpha X_n$,  for all $2\leq i\leq n-1$. Suppose that $2<i<n-i+1$. Arguing as before, the Jacobi identity gives
\begin{align*}
[X_i,X_{n-i+1}]+[X_{i-1},X_{n-i+2}]&=[[X_1,X_{i-1}],X_{n-i+1}]\\
&\quad+[X_{i-1},[X_1,X_{n-i+1}]]\\
&=[X_1,[X_{i-1},X_{n-i+1}]].
\end{align*}
 By the inductive hypothesis, $[X'_{i-1},X'_{n-i+1}]=0$, so $[X_{i-1},X_{n-i+1}]\in \g_{n}$, and therefore $[X_1,[X_{i-1},X_{n-i+1}]]=0$. Thus,
$[X_i,X_{n-i+1}]=-[X_{i-1},X_{n-i+2}]$. Hence, defining $\alpha$ by the condition  $[X_2,X_{n-1}]=\alpha X_n$, we have $[X_i,X_{n-i+1}]=(-1)^i\alpha X_n$, for all $2\leq i\leq n-1$, as required.

Now suppose that $n$ is odd. As in the even case, $[X_i,X_j]\in \g_{i+j}$ when $i+j\geq n+2$. By the inductive hypothesis, we have $[X_i,X_j]\in \g_{i+j}$, for all $i+j\leq n-1$ and there exist $\alpha,\beta_i\in\R$ with $[X_i,X_{n-i}]=(-1)^i \alpha X_{n-1}+\beta_iX_n$,  for all $2\leq i\leq n-2$. By Lemma \ref{L:br}, there exist $\alpha_i\in\R$ such that
$[X_i,X_{n-i+1}]=\alpha_i X_{n}$,  for all $2\leq i\leq n-2$. We must show that $\alpha=0$ and $\alpha_i=0$ for all $2\leq i\leq n-2$.
For $2<i<n-i+1$, the Jacobi identity gives
\begin{align*}
[X_i,X_{n-i+1}]+[X_{i-1},X_{n-i+2}]&=[[X_1,X_{i-1}],X_{n-i+1}]\\
&\quad +[X_{i-1},[X_1,X_{n-i+1}]]\\
&=[X_1,[X_{i-1},X_{n-i+1}]]\\
&=[X_1,(-1)^{i-1} \alpha X_{n-1}+\beta_{i-1} X_n]\\
&=(-1)^{i-1} \alpha X_{n}.
\end{align*}
Thus $\alpha_i  +\alpha_{i-1}=(-1)^{i-1} \alpha$ and hence
\begin{equation}\label{E:alphai}
\alpha_i=(-1)^{i} (\alpha_2+(2-i)\alpha).
\end{equation}
Notice that as $n$ is odd, $X_i=X_{n-i+1}$ when $i=\frac{n+1}2$; so $\alpha_{\frac{n+1}2}=0$. Thus
\begin{equation}\label{E:alphaalpha}
0=\alpha_2+\frac{3-n}2\alpha.
\end{equation}
Note that
\begin{equation}\label{E:alpha}
[X_2,[X_{\frac{n-1}2},X_{\frac{n+1}2}]]=[X_2,(-1)^{\frac{n-1}2} \alpha X_{n-1}]=(-1)^{\frac{n-1}2} \alpha \alpha_2X_{n}.
\end{equation}
The Jacobi identity also gives
\[
[X_2,[X_{\frac{n-1}2},X_{\frac{n+1}2}]]
=[[X_2,X_{\frac{n-1}2}],X_{\frac{n+1}2}]+[X_{\frac{n-1}2},[X_2,X_{\frac{n+1}2}]].
\]
Unless $2+{\frac{n-1}2}\in\{n,n+1\}$, we have  $[X_2,X_{\frac{n-1}2}]\in\g_{\frac{n+3}2}$ and thus $[[X_2,X_{\frac{n-1}2}],X_{\frac{n+1}2}]=0$, by Lemma \ref{L:br}. Similarly,  unless $2+{\frac{n+1}2}\in\{n,n+1\}$, we have $[X_2,X_{\frac{n+1}2}]\in\g_{\frac{n+5}2}$ and thus $[X_{\frac{n-1}2},[X_2,X_{\frac{n+1}2}]]=0$. Hence $[X_2,[X_{\frac{n-1}2},X_{\frac{n+1}2}]]=0$ except possibly in the following four cases:
\begin{enumerate}[\rm (a)]
\item $2+{\frac{n-1}2}=n$; that is, $n=3$.
\item $2+{\frac{n-1}2}=n+1$; that is, $n=1$.
\item $2+{\frac{n+1}2}=n$; that is, $n=5$.
\item $2+{\frac{n+1}2}=n+1$; that is, $n=3$.
\end{enumerate}
But cases (a), (b) and (d) do not occur, since we have assumed $n>4$. In case (c), we have $n=5$ and so
\[
[X_2,[X_{\frac{n-1}2},X_{\frac{n+1}2}]]=[X_2,[X_2,X_3]].
\]
We have $[X_2,X_3]\in\g_5$ and so $[X_2,X_3]=\gamma X_5$ for some $\gamma\in\R$. But then, by Lemma \ref{L:br},
$[X_2,[X_2,X_3]]=\gamma [X_2,X_5]=0$, since $X_5\in\z(\g)$. Consequently $[X_2,[X_{\frac{n-1}2},X_{\frac{n+1}2}]]=0$ in all cases. So, by
Equation (\ref{E:alpha}), $\alpha \alpha_2=0$. Thus, by Equation (\ref{E:alphaalpha}), $\alpha =\alpha_2=0$. Hence, by Equation (\ref{E:alphai}), $\alpha_i=0$  for all $2\leq i\leq n-2$. This completes the proof of the theorem.
\end{proof}

\bibliographystyle{amsplain}

\end{document}